\documentclass[11 pt,letterpaper] {amsart}
\usepackage{amsxtra,comment}
\usepackage{latexsym, amsmath, amssymb, longtable, booktabs,amscd,microtype,booktabs,cases,tikz-cd}
\usepackage[square]{natbib}
\usepackage{graphicx}
\usepackage{caption}
\usepackage{tcolorbox}
\usepackage{tikz}
\usepackage{dsfont}
\usepackage[all]{xy}
\usepackage{enumerate}
\usepackage{multirow}
\usepackage{ctable}

\usepackage[colorinlistoftodos,prependcaption, textsize=tiny]{todonotes}
\reversemarginpar

\usepackage{hyperref}
\hypersetup{
    colorlinks,
    citecolor=blue,
    filecolor=black,
    linkcolor=red,
    urlcolor=black
}
\bibpunct{[}{]}{,}{n}{}{;} 

\numberwithin{equation}{section}
\usepackage{mathtools}
\DeclarePairedDelimiterX{\infdivx}[2]{(}{)}{%
	#1\;\delimsize\|\;#2%
}

\newcommand{\Z}{\mathbb{Z}}

\newcommand{\de}{\delta}

\newcommand{\C}{\mathbb{C}}

\newcommand\A{\mathbb{A}}

\newcommand\GL{{\mathrm{GL}}}

\newcommand{\Q}{\mathbb{Q}}

\newcommand{\sym}{\mathrm{sym}}
\newcommand{\mco}{\mathcal O}
\newcommand\Ind{{\mathrm{Ind}}}

\newtheorem{thm}{Theorem}[section]
\newtheorem{theorem}[thm]{Theorem}
\newtheorem{cor}[thm]{Corollary}

\newtheorem{prop}[thm]{Proposition}
\newtheorem{lemma}[thm]{Lemma}
\theoremstyle{definition}
\newtheorem{definition}[thm]{Definition}
\newtheorem{remark}[thm]{Remark}

\theoremstyle{definition}

\theoremstyle{remark}

\theoremstyle{remark}

\makeatletter
\def\imod#1{\allowbreak\mkern10mu({\operator@font mod}\,\,#1)}
\makeatother

\title{Two properties of symmetric cube transfers of modular forms}
\author{Debargha Banerjee, Tathagata Mandal, and Sudipa Mondal}
\address{Debargha Banerjee, debargha@iiserpune.ac.in, IISER Pune, India}
\address{Tathagata Mandal, math.tathagata@gmail.com, ISI Kolkata, India}
\address{Sudipa Mondal, sudipa.mondal123@gmail.com, HRI, Prayagraj, India}
\thanks{The first author is partially supported by the SERB grant 
	  CRG$/2020/000223$. The second author 
	acknowledges NBHM postdoctoral fellowship at ISI Kolkata. The third author 
	is supported by the institute postdoctoral fellowship at HRI Prayagraj. The authors are indebted to ICTS/ecl2022/8 for a conference at ICTS from which the work of the paper started. 
}
\begin{document}
	\begin{abstract}
		In this article, we study two important properties of $\sym^3$ transfers of the automorphic representation $\pi$ associated to a modular form. First we  compute the conductor of $\sym^3(\pi)$. Then we detect the types of local automorphic representations at bad primes by the variation of the epsilon factors of  symmetric cube transfer of the representation $\pi$ attached to a cusp form $f$. Here we twist the modular forms  by a specific quadratic character. From this variation number, for each prime $p$, we classify all possible types of symmetric cube transfers of the local representations $\pi_p$. For $\sym^3$ transfer, the most difficult prime is $p=3$. 
 \end{abstract}
	
	\maketitle
	\setcounter{tocdepth}{1}
\tableofcontents{}
\section{Introduction}
Let $f\in S_k(N, \epsilon)$ be a cusp form of level $N=\prod_p p^{N_p}$ with the corresponding automorphic representation $\pi=\otimes'\pi_p$. Denote by $\pi_p$ the local representation attached to $f$ at a prime $p$. 
By the seminal work due to Kim-Shahidi  \cite{MR1923967}, we can do the $\sym^3$ transfer to get an automorphic representation of $\GL_4$. The aim of this article is to explore the properties of this automorphic representation $\sym^3(\pi)$ that can be 
detected by the modular form we are starting with. Namely, we will be interested in two properties of these symmetric cube transfers:
\begin{itemize}
\item 
Conductors of these symmetric cube transfers in terms of properties of modular forms. 
\item 
Types of automorphic representations at all primes. Most importantly at bad primes (primes dividing the level $N$). 

\end{itemize}

For each prime $p \mid N$, we compute the conductor of $\sym^3(\pi_p)$ depending upon the nature of the representation $\pi_p$ and we have shown this can be determined in terms of $N_p$, the exact power of $p$ that divides $N$ [cf. Propositions \ref{aprin}, \ref{prop7}, \ref{a2}]. 
We call $f$ to be minimal if its associated representation $\pi$ is primitive in the sense of \cite[Definition 2.5]{LW}.
Combining these results, under some mild assumptions stated in Hypothesis {\bf (H)} in \S \ref{symconductor},  we determine the conductor of the symmetric cube transfer of a minimal newform $f$ in the theorem below, where $a(\sym^3(\pi)):=\prod_p a(\sym^3(\pi_p))$.

Denote by $P_i, S_i, SP, SC$ certain sets of primes dividing level $N$, defined in \ref{sets} and the number $e_\varkappa$ is defined in \eqref{echi}. We prove the following result:
\begin{theorem} \label{conductortheorem}
	Let $\pi$ be the cuspidal automorphic representation attached to a minimal newform $f \in S_k(N, \epsilon)$ satisfying Hypothesis {\bf (H)}. Then the conductor of $\sym^3(\pi)$ is given by
	\begin{eqnarray}  \nonumber
	a \left( \sym^3(\pi) \right)= N  \prod_{p\in SP \cup P_1} p^{2N_p} \prod_{p\in P_2} p^{2N_p-1}  \prod_{p\in P_3 \cup S_2} p^{N_p}
	\prod_{2 \in SC} 2^{N_2} \prod_{3\in SC} 3^{e_\varkappa}.
	\end{eqnarray} 
\end{theorem}

To prove the above theorem, it is important to find a relation between $a(\pi_p)=N_p$ and $a (\sym^3(\pi_p))$. Note that if the conductor of a character contributes in $a(\pi_p)$, then the conductors of its square and cube involve in determining $a (\sym^3(\pi_p))$. Beacuse of this fact, while computing the conductors of $\sym^3(\pi_p)$ for $p=2$ or $3$ the challenging part is to evaluate $a(\chi^i), i= 2, 3$, for a character $\chi$ of $\Q_p^\times$ in terms of $a(\chi)$. We have shown that it can be determined as $a(\chi)-1$ in most cases, see  Propositions \ref{chiq} and \ref{chi2}. For supercuspidal primes other than $3$, we have established that the conductor of $\sym^3(\pi_p)$ is two times of that of $\pi_p$, see Propositions \ref{prop7} and \ref{a2}.

To estimate the number of cuspidal automorphic representations of $\GL_4(\A_\Q)$ obtained from $\GL_2(\A_\Q)$ by symmetric cube transfer, it is necessary to establish a connection between the conductors of $\pi_p$ and $\sym^3(\pi_p)$, where $\pi_p \cong \Ind_{W(K)}^{W(\Q_p)} \varkappa$ with $K/\Q_p$ quadratic, see \cite[Prop. 12]{BM}. In {\it loc. cit.}, it is assumed that $\pi_p$ has trivial central character with $\varkappa^6 \neq 1$. Our Theorem \ref{conductortheorem} can be useful to remove such assumptions. This method can also be helpful to compute $a(\sym^n(\pi_p))$ to generalize the results of \cite{BM} for $\sym^n$ transfer. 
	
In the context of conductor of the symmetric cube transfer of a minimal newform with trivial nebentypus, we mention that \cite[Corollary 3.2.3]{manami} determines the conductor of $\sym^3(\pi_E)$ where $\pi_E$ is the global representation attached to a non-CM elliptic curve $E/\Q$ with the assumption that $E$ has either good or multiplicative reduction at $p=2$. That is, $\pi_E$ cannot be of supercuspidal type at $p=2$, but, in our Theorem \ref{conductortheorem} $p=2$ can be a dihedral supercuspial prime for $\pi_2$. In Remark \ref{r1}, we have shown that the conductor formula of {\it loc. cit.} coincides with ours in the trivial nebentypus setting.
	
	
\smallskip

We now proceed to state our second result. For the automorphic representation $\pi_p$ with bad primes $p \mid N$, the local representations can be of 
ramified principal series, special or of supercuspidal type. To detect these types is of pertinent interest. For example, the ramification formulae of the local endomorphism algebras of motives attached to non-CM Hecke eigenforms for primes dividing level are given in \cite{MR3096563}, \cite{MR3958053} depending upon the types of $\pi_p$. These types can be determined using the algorithms given in \cite{LW}. In loc. cit,  the authors use results of Casselman on the structure of irreducible representations of $\GL_2(\Q_p)$ together with the theory of the homology of modular curves (namely  modular symbols).

\smallskip

Pacetti \cite{pacetti} and the first two authors \cite{bm} of the present article studied how the local epsilon factors associated to $f$ change while twisting by a quadratic character $\chi$ and determined the type of the local automorphic representations from these ratios. 
In this article, we study the same for the symmetric cube transfers of a modular form and determine the type of $\sym^3(\pi_p)$ in terms of the varience of the local epsilon factors while twisting by $\chi$ [cf. \S \ref{aim}].

\smallskip

%
%

As supercuspidal types are important to detect, we mention that one can determine the case where $\pi_p$ is supercuspidal by the variation of the local epsilon factors by twisting with respect to a certain set of characters \cite[local inverse theorem, \S 27] {bh} and it is less convenient from computational perspective \cite{LW}. 

\smallskip

In this present paper, we manage to do the same for $\sym^3(\pi_p)$ by suitable quadratic twists. For $p$ odd and $\pi_p=\Ind_{W(K)}^{W(\Q_p)} \varkappa$ supercuspidal with $K/\Q_p$ quadratic, $\sym^3(\pi_p)$ can be either of Type I, II, or III [cf. Section \ref{symtype}]. For more details of these types, we refer to \cite[Chapter 2]{BR}. In this article, we classify the quadratic extensions $K$ of $\Q_p$ and these types in Theorem \ref{mainthm2} below in terms of the ratio of the local epsilon factors while twisting by $\chi_p$. 

\smallskip

While computing the varience of local epsilon factors, we use the same additive character on both the components [cf. Section \ref{aim}] and keep on using an additive character of conductor $-1$ for each local type of $\pi_p$. 
Write $N=p^{N_p}N'$ and $\epsilon=\epsilon_p \cdot \epsilon'$. Let $a(\epsilon_p)=p^{C_p}$ for some $C_p \leq N_p$ where the conductor $a(\chi)$ of a character $\chi$ is defined in \S \ref{characters}. When $\pi_p$ is supercuspidal, we discuss the possibilities of the Types I, II and III of $\sym^3(\pi_p)$ in Propositions \ref{types} and \ref{a2} depending upon $N_p$ and $C_p$.
Let $M'$ denote the prime-to-$p$ part of $a(\sym^3(\pi))$ as determined in Theorem \ref{conductortheorem}. For a $p$-minimal form (see \S \ref{aim}), the theorem below [Corollary \ref{mainkoro}] detects the types of $\sym^3(\pi_p)$ at each $p \mid N$.
\begin{theorem} \label{mainthm2}
	Let $\sym^3(\pi_p)$ be the symmetric cube transfer of $\pi_p$, the local representation at an odd prime $p$ attached to a $p$-minimal form $f$. 
	Then $\sym^3(\pi_p)$ is a special representation if $N_p=1$ and $C_p=0$ and it is of principal series type if $N_p \geq 1$ with $N_p=C_p$.  
	
	Otherwise, it is either of Type I, II or III.  In this case, $2 \leq N_p > C_p$ and $\sym^3(\pi_p)=\Ind_{W(K)}^{W(\Q_p)} \varkappa ~~ \oplus ~~$ $\Ind_{W(K)}^{W(\Q_p)} (\varkappa \epsilon_p')$ with $K/\Q_p$ quadratic and $\epsilon_p'=\epsilon_p^{-1}\circ N_{K/\Q_p}$. If $N_p$ is even, then $K$ is the unique unramified quadratic extension of $\Q_p$. In this case, for $C_p \neq \frac{N_p}{2}$, if $p \geq 5$ or $p=3$ with $a(\varkappa^3)>1$, we have
	
	\begin{table}[h!]
		\begin{center}
			\begin{tabular}{|c|c|c|}
				\hline
				{\bf Type of {\boldmath $\sym^3(\pi_p)$}}  & {\bf  Condition} \\
				\hline
				\multirow{2}{*}{Type I or II} &
				\multirow{2}{*}{$\varepsilon(\sym^3(\pi) \otimes \chi_p)= \chi_p(M') \chi_p'(s) \varepsilon(\sym^3 (\pi))$} \\ & {} \\
				\hline
				\multirow{2}{*}{Type III} &
				\multirow{2}{*}{$\varepsilon(\sym^3(\pi) \otimes \chi_p)= \chi_p(dM') \chi_p'(s_2)\varepsilon(\sym^3 (\pi))$} \\ & {} \\
				\hline
			\end{tabular}
			\vskip 1mm
		\end{center}
	\end{table}	
	\noindent If $N_p$ is odd, then $K/\Q_p$ is ramified quadratic. In this case, ${\sym^3}(\pi_p)$ cannot be of Type I (resp. Type III) if $C_p \leq 1$ (resp. if $p \geq 5$). In particular, for primes $p\geq 5$ with $C_p \leq 1$, ${\rm sym^3}(\pi_p)$ is always of Type II. For $p=3$ with $N_3$ odd, we have the below classification of $\sym^3(\pi_p)$ and $K/\Q_p$. 
	\begin{itemize}
		\item 
		If $a(\varkappa^3)=0$, then $\sym^3(\pi_3)$ is of Type III and  $\varepsilon(\sym^3(\pi) \otimes \chi_3)=  \chi_3(M')\varepsilon(\sym^3 (\pi))$.  
		\item
		If $a(\varkappa^3)=1$, then $\sym^3(\pi_3)$ is either Type I or II.
		\item 
		If $a(\varkappa^3) \geq 2$ is even, then it is always of either Type I or II and we have $\varepsilon(\sym^3(\pi) \otimes \chi_3)= \chi_3(M')\varepsilon(\sym^3 (\pi))$.
		\item 
		If $a(\varkappa^3) \geq 3$ is odd, then the corresponding ramified extensions are determined in Table \ref{table555} depending upon the types.
	\end{itemize}
\end{theorem}
We mention that if $\pi_p$ is of principal series type with $N_p=1$ or of supercuspidal type with $N_p=2$ for an odd prime $p$, we use Davenport-Hasse theorem, Gross-Koblitz formula and Stickelberger’s theorem to compute the ratio $\varepsilon_p$ [cf. Equ. \ref{ratio}] and this ratio is determined in terms of Morita's $p$-adic gamma function in most cases; see Theorems \ref{prin} and \ref{supthm1}. 

For $p=2$, the classification of the symmetric cube transfer of the corresponding local Weil-Deligne representation is given in Corollary \ref{mainkorop=2} if it is non-dihedral. If $N_2=2$, we compute epsilon factors by evaluating characters at the representatives of the quotient group of units (see Lemmas \ref{chip'}, \ref{chip'11}).

\smallskip

\smallskip 
	
\section{Preliminaries on characters and Gauss sums}
\subsection{Characters} \label{characters}
Let $K$ be a non-archimedean local field of characteristic zero. Let $\mco_K$, $\pi_K$, $\mathfrak{p}_K$, $\kappa_K$ be its ring of integers, a uniformizer, the maximal ideal and the residue field of $K$ respectively. We denote by $U_K^n=1+\mathfrak{p}_K^n$ the $n$-th principal units of $K$.
	
\begin{definition}
	The conductor $a(\chi)$ of a multiplicative character $\chi$ of $K^\times$ is the smallest positive integer $n$ such that ${\chi}_{|_{1+\mathfrak{p}_K^n}} = 1$. Sometimes, we write $a(\chi)=\mathfrak{p}_K^n$.
		
	The conductor $n(\phi)$ of a non-trivial additive character $\phi$ of $K$ is the smallest positive integer such that ${\phi}_{|_{\mathfrak{p}_K^{n(\phi)}}} = 1$. 
\end{definition}
The character $\chi$ is called unramified if $a(\chi)=0$ and tamely ramified if $a(\chi)=1$. If $\chi$ has character conductor $t \geq 1$, then we have a non-trivial induced character $\widetilde{\chi}: \mco_K^\times/U_K^t \to \C^\times$ with order denoted by $\circ(\widetilde{\chi})$. Note that, for $i\geq0$, $a(\chi^i)\leq a(\chi)$. The set of all multiplicative (resp. additive) characters of $K$ is denoted by $\widehat{K^\times}$ (resp. $\widehat{K}$). 
	
Our goal is to find the conductor of $\chi^3$ as we want to compute $a(\sym^3(\pi))$.
\begin{prop} \label{chiq}
	Let $\chi$ be a character of $\Q_p^\times$ with conductor $a(\chi)$. If $p \geq 5$ is a prime number, then $a(\chi^3)=a(\chi)$ provided that $a(\chi) \neq 1$. For $p=3$, we have
	\[
	a(\chi^3) = \begin{cases}
	a(\chi) & \text{ if } a(\chi)=0 \text{ or } 1 \\
	0  & \text{ if the order of } \widetilde{\chi} \text{ is } 3 \text{ with } a(\chi)=2 \\
	a(\chi)-1 & \text{ otherwise. }
	\end{cases} \]
\end{prop}

\begin{proof} 
	The case $p \geq 5$ follows from the fact that both $\widetilde{\chi}$ and $(\widetilde{\chi})^3$ have the same order and $(\widetilde{\chi})^3=\widetilde{\chi^3}$. For the remaining of the proof, we assume $p=3$.
	
	\smallskip

	If $\chi$ is unramified, the assertion follows from the fact that $\chi|_{\Z_p^\times}=1 ~~\Rightarrow \chi^3|_{\Z_p^\times}=1$. In the case of $a(\chi)=1$, $\circ(\widetilde{\chi})=2$. Thus, $\widetilde{\chi^3}$ has order $2$, and so $a(\chi^3)=1$. For $a(\chi)=2$, $\widetilde{\chi}$ restricted to $U_{\Q_p}^1/U_{\Q_p}^2$ has order $3$. Hence, $\circ(\widetilde{\chi})=3$ or $6$, and so $a(\chi^3)=0$ or $1$ accordingly.
	
	\smallskip
	

	\smallskip
		
	Next, we assume $a(\chi)=t$ where $t \geq 3$. We claim that $\circ(\widetilde{\chi}) \neq 3$. If $2 \mid \circ(\widetilde{\chi})$, then we are done. So suppose $2 \nmid \circ(\widetilde{\chi})$. If $\circ(\widetilde{\chi}|_{U_{\Q_p}^{1}/U_{\Q_p}^t})=3$, then as $U_{\Q_p}^{1}/U_{\Q_p}^t$ is a cyclic group of order $3^{t-1}$ and $U_{\Q_p}^{t-1}/U_{\Q_p}^t$ is its subgroup of order $3$, we conclude that $\widetilde{\chi}|_{U_{\Q_p}^{t-1}/U_{\Q_p}^t}=1$, contradicting $a(\chi)=t$. In other words, the claim follows.
	If $a(\chi^3)=s$, then $s \leq t$. As $a(\chi)=t$, we have a non-trivial induced character $\widetilde{\chi}: \Z_p^\times/U_{\Q_p}^t \to \C^\times$. Note that $\Z_p^\times/U_{\Q_p}^t$ is a cyclic group of order $2\cdot 3^{t-1}$; so assume that $\Z_p^\times/U_{\Q_p}^t=\langle g \rangle$. Also, the subgroup $U_{\Q_p}^s/U_{\Q_p}^t$ of $\Z_p^\times/U_{\Q_p}^t$ is cyclic of order $3^{t-s}$ and it is generated by $g^{2\cdot 3^{s-1}}$. Now since $\chi^3$ is trivial on $U_{\Q_p}^s$, we have that
	$\widetilde{\chi^3}(g^{2\cdot 3^{s-1}}) =1 ~~~~ \Rightarrow \widetilde{\chi}(g^{2\cdot 3^s})=1.$
	This shows that $\widetilde{\chi}$ is trivial on $U_{\Q_p}^{s+1}/U_{\Q_p}^t$. Hence, $t\leq s+1$. 
	In other words,
	\begin{eqnarray} \label{achi3}
		a(\chi^3)  \leq a(\chi) \leq a(\chi^3)+1.
	\end{eqnarray}
	Note that $a(\chi)=t ~~\Rightarrow \widetilde{\chi}|_{U_{\Q_p}^{t-1}/U_{\Q_p}^t}\neq 1$; but notice that $\widetilde{\chi^3}|_{U_{\Q_p}^{t-1}/U_{\Q_p}^t}= 1$. This implies that $a(\chi^3)=a(\chi)$ can never happen. Hence from Equ. \ref{achi3} we conclude that $a(\chi) = a(\chi^3)+1$. 
	This completes the proof.
\end{proof}
	
	\begin{prop} \label{chi2}
		Let $\chi$ be a character of $\Q_2^\times$ with conductor $a(\chi)$. Then we have:
		\[
		a(\chi^2) = \begin{cases}
		0 & \text{ if } a(\chi)=2, 3 \\
		a(\chi)-1 & \text{ otherwise. }
		\end{cases} \]
	\end{prop}
	
	\begin{proof}
		The first case is clear. So we assume that $a(\chi)=t \geq 4$. As $\Z_2^\times=1+2\Z_2$, the induced character $\widetilde{\chi}: \Z_2^\times/U_{\Q_2}^t \cong \Z/2 \times \Z/2^{t-2} \to \C^\times$ is non-trivial. We claim that $U_{\Q_2}^2/U_{\Q_2}^t$ is a cyclic group of order $2^{t-2}$. Indeed, using \cite[Prop. 5.5]{neukirch} we have $U_{\Q_2}^n \cong (2)^n$ for $n \geq 2$, where $(2)=2\Z_2$. Thus, $U_{\Q_2}^2/U_{\Q_2}^t \cong (2)^2/(2)^t$ and it contains the element $2^2+(2)^t$ of order $2^{t-2}$, as desired. 
		
		Let $a(\chi^2)=s$. As $U_{\Q_2}^{t-1}/U_{\Q_2}^t$ is a subgroup of order $2$ of the cyclic group $U_{\Q_2}^{2}/U_{\Q_2}^t$, we deduce that $\circ(\widetilde{\chi}|_{U_{\Q_2}^{2}/U_{\Q_2}^t})=2 \Rightarrow \widetilde{\chi}|_{U_{\Q_2}^{t-1}/U_{\Q_2}^t}=1$; contradicting $a(\chi)=t$. Thus, we obtain $s \geq 2$. 
		Hence, $U_{\Q_2}^s/U_{\Q_2}^t$ is cyclic of order $2^{t-s}$. Now by the same argument used in Prop.\ref{chiq} for $p=3$, we have $a(\chi^2)=a(\chi)-1$.
	\end{proof}

	For $K/\Q_3$ quadratic and $\chi \in \widehat{K^\times}$, we now determine $a(\chi^3)$.  As there is no proper relation between the conductors of $\chi$ and $\chi^3$, we will proceed as follows. Let $a(\chi)=t$.
	
	For $K/\Q_3$ ramified, $\mco_K^\times/U_K^t$ is a group of order $2\cdot 3^{t-1}$, and $\chi$ is completely determined if it is determined on $\{a_1^1, a_1^2\} \cup \{a_j^1, a_j^2, a_j^3 : 1<j\leq t\}$ with $a_j^i \neq a_k^l$, $j \neq k$. Here, the elements $a_j^i$ are given as follows: \begin{center}$U_K^{j-1}/U_K^j =\{a_j^1U_K^j, a_j^2 U_K^j, a_j^3 U_K^j\}$ if $j \neq 1$; otherwise $\{a_1^1U_K^1, a_1^2 U_K^1\}$. \end{center}
	Similarly, if $\chi$ is a character of $K^\times$ of conductor $t$ where $K/\Q_3$ is unramified, then it is completely determined if its values on $\{a_1^1, \cdots, a_1^8\} \cup \{a_j^1, a_j^2, \cdots, a_j^9 : 1<j\leq t\}$ with $a_j^i \neq a_k^l$, $j \neq k$, are known. The elements $a_j^i$ are given as follows: \begin{center}$U_K^{j-1}/U_K^j =\{a_j^1U_K^j, a_j^2 U_K^j, \cdots, a_j^9 U_K^j\}$ if $j \neq 1$; otherwise $\{a_1^1U_K^1, \cdots, a_1^8 U_K^1\}$. \end{center}
	If $a(\chi)=t \geq 2$, we define a positive integer $f_\chi$ as follows:
	\begin{equation} \label{fchi}
	f_\chi:= \text{ The largest integer between } 1 \text{ and } t \text{ such that } \widetilde{\chi}(a_{f_\chi}^i) \text{  has order } \neq 3 \text{ for some } i.
	\end{equation}
	The proposition below connects the conductor of $\chi$ and $\chi^3$. Its proof follows from definition of conductors.
	\begin{prop} \label{chi33}
		Let $K/\Q_p$ be a quadratic extension and $\chi \in \widehat{K^\times}$. For $p \geq 5$, $\chi^3$ is unramified if $a(\chi)=1$ with $\circ(\widetilde{\chi})=3$; otherwise $a(\chi^3)=a(\chi)$. If $p=3$, then  
		\begin{equation} \label{chi3}
		a(\chi^3) = \begin{cases}
		a(\chi) & \text{ if } a(\chi)=0 \text{ or } 1 \\
		0  & \text{ if the order of } \widetilde{\chi} \text{ is } 3 \\
		f_\chi & \text{ otherwise. }
		\end{cases} \end{equation}
		Moreover, $f_\chi \neq a(\chi)$ if $K/\Q_3$ is ramified.  
	\end{prop}

	\subsection{Gauss Sum} \label{Gauss}
	Let $\mathbb{F}_{p^r}$ be a field of order $p^r$. The classical Gauss sum $G(\chi,\phi)$ associated to $\chi \in \widehat{\mathbb{F}_{p^r}^\times}$ and $\phi \in \widehat{\mathbb{F}_{p^r}}$ is defined by
	$
	G(\chi,\phi)=\sum_{x \in \mathbb{F}_{p^r}^\times}^{} \chi(x) \phi(x).
	$
	Denote it by $G_r(\chi)$ when $\phi$ is fixed. If $\chi$ is a character of $\mathbb{F}_p^\times$ of order $k$, then by Gross-Koblitz formula we deduce \cite[Corollary~$3.1$]{MR1014384}:
	\begin{eqnarray} \label{GKcoro}
	G_1(\chi^a)=(-p)^{a/k} 
	\Gamma_p \big( \frac{a}{k} \big).
	\end{eqnarray}
	where $\Gamma_p$ is the $p$-adic gamma function\cite[Chapter~$7$]{MR1760253}.
	
	For $\chi \in \widehat{\mathbb{F}_p^\times}$ and 
	$\phi \in \widehat{\mathbb{F}_p}$, let $\chi'=\chi \circ N_{\mathbb{F}_{p^r}|\mathbb{F}_p}$
	and $\phi'=\phi \circ \mathrm{Tr}_{\mathbb{F}_{p^r}|\mathbb{F}_p}$ denote their lifts
	to $\mathbb{F}_{p^r}$. By Davenport-Hasse theorem, 
	$G(\chi',\phi')=(-1)^{r-1}G(\chi,\phi)$ \cite[Theorem~$11.5.2$]{MR1625181}.
	

	\section{Local Langlands correspondence and symmetric cube transfers}
	\subsection{Local Langlands correspondence and $L$-parameter}
	Let the residue field $\kappa_K$ has characteristic $p > 0$ with $q = | \kappa_K |$. We denote by $W(K)$ the Weil group of $K$. Consider a continuous homomorphism i.e. a representation of $W(K)$ given by $\phi : W(K) \rightarrow {\rm GL}(V)$, where $V$ is a finite dimensional vector space. Let $\omega$ be a character of $W(K)$ defined by
	$g: x \mapsto x^{\omega(g)} \,\, \forall~ x \in \overline{\kappa}_K.
	$
	If there is a nilpotent endomorphism $N$ of $V$ which satisfies
	$$\phi(g)N\phi(g)^{-1}= \omega(g)N \quad \forall ~g \in W(K),$$ then the pair $(\phi,N)$ is said to be a Weil-Deligne representation. See \cite{tate} for more details.\medskip

	By the local Langlands correspondence, a two-dimensional Weil-Deligne representation $(\phi,N)$ corresponds uniquely to an irreducible admissible representation $\pi$ of ${\rm GL}_2(K)$ and vice-versa. We refer to $(\phi,N)$ as the local parameter (or the $L$-parameter) of $\pi$. Now, we consider the case where $K=\Q_p$.
	\begin{enumerate}
		\item 
		For a principal series representation $\pi=\pi(\mu_1,\mu_2)$, where $\mu_1,\mu_2$ are characters of $\Q_p^\times$ with $\mu_1\mu_2^{-1}\neq |\cdot|^{\pm1}$, the $L$-parameter of $\pi$ is given by \begin{equation}
		\phi(x)=\begin{bmatrix} \mu_1(x) & \\ & \mu_2(x) \end{bmatrix},~ x \in W(\Q_p) \text{ and } N=0. 
		\end{equation}
		\item
		The $L$-parameter of a special representation $\pi=\mu \otimes {\rm St}_2$, where $\mu$ is a character of $\Q_p^\times$, is given by \begin{equation} \label{spl}
		\phi(x)=\begin{bmatrix} \mu(x)|x|^{\frac{1}{2}} & \\ & \mu(x)|x|^{-\frac{1}{2}} \end{bmatrix},~ x \in W(\Q_p) \text{ and } N=\begin{bmatrix} 0 & 1 \\ 0 & 0 \end{bmatrix}. 
		\end{equation}
		\item
		For a supercuspidal representation $\pi$ of ${\rm GL}_2(\Q_p)$, the local parameter of $\pi$ is an irreducible representation of $W(\Q_p)$ with $N=0$.
		
		For $p \geq 3$, every supercuspidal representation $\pi$ of ${\rm GL}_2(\Q_p)$ is a dihedral supercuspidal representation (see \cite{bump}) i.e., there exists a character $\varkappa$ of an index two subgroup $W(K)$ of $W(\Q_p)$ such that $\pi$ has a local parameter $\phi$ of the form
		$\phi= {\rm Ind}^{W(\Q_p)}_{W(K)} ~ \varkappa$
		with  $\varkappa \neq \varkappa^{\sigma},$  where $K/\Q_p$ is a quadratic extension and $\sigma$ is a non-trivial element in $W(\Q_p)\setminus W(K)$. Here,  $\varkappa^{\sigma}$ is defined as $\varkappa^{\sigma}(x)=\varkappa(\sigma x \sigma^{-1})$ for $x\in W(K)$. The character $\varkappa$ can be considered as a character of $K^\times$ via the isomorphism $W(K)^{\text{ab}} \cong K^\times$. With respect to a suitable basis, $\phi$ has the following matrix form.   
		\begin{equation} \label{LLPsup}
		\phi(x)=\begin{bmatrix} \varkappa(x) & \\ & \varkappa^{\sigma}(x) \end{bmatrix},~ x \in W(K) \text{ and } \phi(\sigma)=\begin{bmatrix} & 1 \\ \varkappa(\sigma^2) & \end{bmatrix}. 
		\end{equation}
		If the conductor of $\varkappa$ is $a(\varkappa)$, then the conductor $a(\phi)$ of $\phi$ is given by
		\begin{equation} \label{indconductor}
		a(\phi)= {\rm dim}({\varkappa})~v(d_{K/\Q_p}) + f_{K/\Q_p}~ a(\varkappa),
		\end{equation}
		where $d_{K/\Q_p}$ is the discriminant of the field extension $K/\Q_p$ and $f_{K/\Q_p}$ is the residue class degree (see \cite{serre} for more details).
	\end{enumerate}

	\subsection{Symmetric Cube Transfers} \label{sym^3}
	Let $F$ be a number field and $\pi=\otimes'_{v \leq \infty}\pi_v$ be a cuspidal automorphic representation of ${\rm GL}_2(\mathbb{A}_F)$ where $\mathbb{A}_F$ is the ring of ad\`eles. Consider the $3$-rd symmetric power ${\rm {sym}}^3: \text{GL}_2(\mathbb{C}) \to \text{GL}_4(\mathbb{C})$ of the standard representation of ${\rm GL}_2$ which has the following representation, 
	\begin{eqnarray} \label{symcube}
	\text{sym}^3 \left(\begin{bmatrix} a & b \\ c & d \end{bmatrix} \right) =  \begin{bmatrix} a^3 & a^2b & ab^2 & b^3 \\ 3a^2c & 2abc+a^2d & 2abd+b^2c & 3b^2d \\ 3ac^2 & 2acd+bc^2 & 2bcd+ad^2 & 3bd^2 \\ c^3 & c^2d & cd^2 & d^3 \end{bmatrix}.
	\end{eqnarray}
	
	Let $\phi_v$ be the two dimensional representation of the Weil-Deligne group attached to $\pi_v$. Then ${\rm sym^3}(\phi_v)= {\rm sym}^3\circ \phi_v$ is a four dimensional representation of the Weil-Deligne group. Let ${\rm sym}^3(\pi_v)$ be the irreducible admissible representation of ${\rm GL}_4(F_v)$ attached to ${\rm sym}^3(\phi_v)$ by the local Langlands correspondence. Set ${\rm sym}^3(\pi):=\otimes_v'{\rm sym}^3(\pi_v)$ which we call the symmetric cuble of $\pi$. For more details, we refer \cite{MR1923967}. 
	
	In \cite{MR1809630}, Kim-Shahidi have proven the equality of the root number attached to the symmetric cube transfer of a cusp form with that of its corresponding Weil-Deligne representation via local Langlands correspondence.
	

	\section{Epsilon factors and  variations}  
	The local $\varepsilon$-factor associated to a non-trivial character $\chi$ of $F^\times$
	and a non-trivial character $\phi$ of $F$ is defined as follows \cite[p. $5$]{RPL}:
	\[
	\varepsilon(\chi, \phi, c) = \chi(c)
	\frac{\int_{\mco_F^\times}^{} \chi^{-1}(x) \phi(\frac{x}{c}) dx}
	{|\int_{\mco_F^\times}^{} \chi^{-1}(x) \phi(\frac{x}{c}) dx|},
	\]
	where $c \in F^\times$ has valuation $a(\chi)+n(\phi)$.
	Here, we consider the normalized Haar measure $dx$ on $F$.
	By \cite[p. $94$]{MR0457408},
	$\varepsilon(\chi, \phi, c) 
	=q^{- \frac{a(\chi)}{2}} \chi(c) \tau(\chi, \phi)$,
	where
	$\tau(\chi, \phi)=  \sum_{x \in \frac{\mco_F^\times}{U_F^{a(\chi)}}}^{} \chi^{-1}(x) \phi(\frac{x}{c})$,
	the local Gauss sum associated to $\chi$
	and $\phi$. As one can show that $\varepsilon(\chi, \phi, c)= \varepsilon(\chi, \phi, cu)$ for a unit $u$, we write $\varepsilon(\chi, \phi, c)=\varepsilon(\chi, \phi)$ for simplicity. 
	
	\begin{remark} \label{R1}
		(1) Let $\chi$ be unramified. Thus, $c$ has the valuation $n(\phi)$ and so $\varepsilon(\chi,\phi,c)=\chi(c)$.
		
		\smallskip 
		
		\noindent
		(2) If $a(\chi)=1$, then $\widetilde{\chi}:=\chi^{-1}|_{\mco_F^\times} \in \widehat{k_F^\times}$ and after choosing an additive character $\phi$ of conductor $-1$, the local Gauss sum coincides with the well-known classical Gauss sum.
	\end{remark}
	
	We now list some basic properties of local $\varepsilon$-factors which can be found
	in \cite{MR546607}.
	\begin{enumerate}
		\item [($\epsilon 1$)]
		$\varepsilon(\chi, \phi_a)=\chi(a)|a|_F^{-1} \varepsilon(\chi,\phi)$,
		where $a \in F^\times$ and $\phi_a(x)=\phi(ax)$. 
		Here, $|\,\,|_F$ denotes the absolute value of $F$.
		\item  [($\epsilon 2$)]
		$\varepsilon(\chi\theta, \phi)=\theta(\pi_F)^{a(\chi)+n(\phi)}
		\varepsilon(\chi,\phi)$, where $\theta$ is an unramified character of $F^\times$.
		\item  [($\epsilon 3$)]
		$\varepsilon \big(\Ind_{W(F)}^{W(\Q_p)} \rho, \phi \big)
		=\varepsilon \big(\rho, \phi \circ \mathrm{Tr}_{F/\Q_p} \big)$,
		where $\rho$ denotes a virtual zero dimensional representation 
		of a finite extension $F/\Q_p$.
	\end{enumerate}
	
	The canonical additive character $\phi_p$ of $\Q_p$ is defined by $\phi_p(x)=e^{2\pi i \{x\}_p}$ where $\{x\}_p$ denotes the fractional part of $x \in \Q_p$. 
	
	\begin{lemma} \label{alpha}
		Let $\alpha$ be a tamely ramified quadratic character of $\Q_p^\times$ for an odd prime $p$.
		Then for an additive character $\phi$ of $\Q_p$ of conductor $-1$, we
		have 
		\begin{eqnarray*} 
			\varepsilon(\alpha,\phi)=
			\begin{cases}
				1, & \quad \text{if} \,\, p \equiv 1 \pmod{4}, \\
				i, & \quad \text{if} \,\, p \equiv 3 \pmod{4}.
			\end{cases}
		\end{eqnarray*}
		Here, $i$ denotes a primitive fourth root of unity. For $p=2$ and $a(\alpha)=2, 3$, we have  
		\begin{equation} \label{ep2}
		\varepsilon(\alpha,\phi) = \begin{cases}
		\frac{i\alpha(2)}{2} & \text{ if } a(\alpha)=2\\
		\frac{\alpha^2(2)}{2}  & \text{ if } a(\alpha)=3.
		\end{cases} \end{equation}
	\end{lemma}
	
	\begin{proof}
		As $a(\alpha)=1$ we have that $\widetilde{\alpha}: \Z_p^\times/(1+p\Z_p) \to \C^\times$ is non-trivial. Thus, $\widetilde{\alpha} \in \widehat{\mathbb{F}_p^\times}$. We denote the canonical additive characters of $\Q_p$ and $\mathbb{F}_p$ by $\phi_p$ and $\widetilde{\phi}_p$ respectively. From the proof of \cite[Lemma 2.3]{bm}, $\phi$ induces the canonical additive character of $\mathbb{F}_p$ choosing $\phi|_{\Z_p}=\frac{1}{p} \cdot \phi_p|_{\Z_p}$. 
		Note that the local Gauss sum $\tau(\alpha, \phi)$ is the same as the classical Gauss sum $G(\widetilde{\alpha}, \widetilde{\phi}_p)$. For odd primes, we have \cite[Theorem 5.15]{MR1294139}
		\begin{eqnarray*} 
			G(\widetilde{\alpha}, \widetilde{\phi}_p)=
			\begin{cases}
				p^{1/2}, & \quad \text{if} \,\, p \equiv 1 \pmod{4} \\
				ip^{1/2}, & \quad \text{if} \,\, p \equiv 3 \pmod{4}.
			\end{cases}
		\end{eqnarray*}
		As
		$
		\varepsilon(\alpha, \phi)=p^{-1/2}\tau(\alpha, \phi)
		$ by the definiton of epsilon factors, we complete the proof of the lemma for odd primes.

		Now, we consider $p=2$. Using the property $(\epsilon 1)$ of epsilon factors, 
		\begin{equation} \label{eqalpha}
		\varepsilon(\alpha, \phi)=\varepsilon(\alpha, \frac{1}{2}\phi_2)=\frac{1}{2\alpha(2)}\varepsilon(\alpha, \phi_2).
		\end{equation}
		Let $a(\alpha)=2$. From \eqref{eqalpha}, $\varepsilon(\alpha, \phi)=\frac{\alpha(2)}{4}\tau(\alpha, \phi_2)$. By the choice of additive charcaters,
		$\tau(\alpha, \phi_p)=\sum_{x \in \Z_2^\times/(1+4\Z_2)} \alpha(x) \phi_p(\frac{x}{4})$ $=\phi_2(\frac{1}{4}) - \phi_2(\frac{3}{4})=e^{\frac{2\pi i}{4}}-e^{\frac{6\pi i}{4}}=2i$. For $a(\alpha)=3$, $\tau(\alpha, \phi_2)=\sum_{x \in \Z_2^\times/(1+8\Z_2)}$  $\alpha(x) \phi_p(\frac{x}{8})$. Note that $1, 3, 5$ and $7$ can be taken as the representatives of $\Z_2^\times/(1+8\Z_2) \cong \Z/2\Z \times \Z/2\Z$. Also, $\alpha$ is quadratic and it has three choices with $\ker(\alpha)$ having two elements. If $\ker(\alpha)=\{1, 5\}$, then $\alpha=1$ on $1+4\Z_2$, contradicting $a(\alpha)=3$. For the remaining two choices of $\alpha$, we compute that $\tau(\alpha, \phi_2)=2\sqrt{2}$. Now, From \eqref{eqalpha}, $\varepsilon(\alpha, \phi)=\frac{\alpha^2(2)}{4\sqrt{2}}\tau(\alpha, \phi_2)=\frac{\alpha^2(2)}{2}$. This completes the proof.
	\end{proof}

	\begin{lemma} \label{alphabeta}
		Let $\alpha$ and $\beta$ be two tamely ramified quadratic characters of $\Q_p^\times$ for an odd prime $p$. Then for an additive character $\phi$ of $\Q_p$ of conductor $-1$,  we have  
		\begin{equation*} 
		\varepsilon(\alpha\beta,\phi)=\frac{1}{(\alpha\beta)(p)}.
		\end{equation*}
	\end{lemma}
	
	\begin{proof}
		As both $\alpha$ and $\beta$ have conductor $1$, both $\widetilde{\alpha}$ and $\widetilde{\beta}$ have order $2$ when restricted to the cyclic group $\Z_p^\times/(1+p\Z_p) \cong (\Z/p\Z)^\times$. Now, computing at the generators we deduce that $\alpha=\beta$ on $\Z_p^\times$. As a result, the product $\alpha\beta$ becomes unramified. Thus, $\varepsilon(\alpha\beta,\phi) = \frac{1}{(\alpha\beta)(p)}$ by Remark \ref{R1}(1).
	\end{proof}


	Let $\phi$ be an additive character of $\mathbb{Q}_p$. Then $\phi_{F}= \phi \circ {\rm{Tr}}_{{F}/\mathbb{Q}_p}$ is a character of ${F}$. 
	
	\begin{lemma} \cite[Lemma 2.4]{bm} 
	\label{relation between multiplicative and additive character}
		Let $\chi\in \widehat{{F}^\times}$ and $\phi_{F}\in \widehat{{F}}$ be two non-trivial characters. Let $r\in \mathbb{N}$ be such that $2r \geq a(\chi)$. Then there is an element $c \in {F}^\times$ with valuation $-(a(\chi)+n(\phi_{{F}}))$ such that 
		$$ \chi(1+x)=\phi_{F}(cx) ~~\forall~ x\in \mathfrak{p}_{{F}}^r.$$
	\end{lemma}
	
	The following theorem is a fundamental result of Deligne \cite[Lemma 4.16]{deligne} about the behaviour of local factors while twisting.	
	
	\begin{theorem}\label{epsilon factor while twisting}
		Let $\alpha, \beta$ be two quasi characters of ${F}^\times$ such that $a(\alpha) \geq 2a(\beta)$. If $\alpha(1+x)=\phi_{F}(cx)$ for $x\in\mathfrak{p}_{{F}}^r$ with $2r \geq a(\alpha)$(if $a(\alpha)=0$ then $c= \mathfrak{p}_{{F}}^{-n(\phi_{F})}$), then $$\varepsilon(\alpha\beta, \phi_{F})= \beta^{-1}(c)\varepsilon(\alpha, \phi_{{F}})$$ where the valuation of $c$ is $-(a(\alpha)+n(\phi_{F}))$.
	\end{theorem}

	\subsection{Variations of epsilon factors} \label{aim}
	Let $f \in S_k(N, \epsilon)$ be a $p$-minimal newform of weight $k$, level $N$ and nebentypus $\epsilon$. We say $f$ is $p$-minimal if the $p$-part of its level is the smallest among all twists of $f$ by Dirichlet characters. Write $N=p^{N_p}N'$ with $p \nmid N'$ and $\epsilon=\epsilon_p \cdot \epsilon'$, where $\epsilon_p$ is the $p$-part and $\epsilon'$ is the prime-to-$p$ part of $\epsilon$. The conductor of $\epsilon_p$ is $p^{C_p}$ for some $C_p\leq N_p$. Let $\pi_p:=\pi_{f,p}$ be the local representation attached to $f$ at $p$. In this article, depending upon the type of $\pi_p$ we will study the variance of the epsilon factor of the symmetric cube transfer of $\pi_p$ while twisting by a quadratic character $\chi_p$. 
	More specifically, we study 
	\begin{eqnarray} \label{ratio}   \varepsilon_p:=\frac{\varepsilon \left(\sym^3(\pi_{f,p}) \otimes \chi_p \right)}{\varepsilon \left(\sym^3(\pi_{f,p}) \right)}.\end{eqnarray}

	\smallskip 
	
	The above character $\chi_p$ is defined as follows. Let $\chi$ denote the quadratic character attached to the quadratic extension of $\Q$ ramified only at $p$. For an odd prime $p$, set $p^*:=\big( \frac{-1}{p} \big) \cdot p$. Then the quadratic extension $\Q(\sqrt{p^*})/\Q$ is ramified only at $p$. By class field theory, 
	the character $\chi$ can be identified with a character of the id\`{e}le group, i.e., 
	characters $\{\chi_q\}_q$ with $\chi_q: \Q_q^\times \to \C^\times$
	satisfying the following conditions:
	\begin{enumerate}
		\item
		For primes $q \neq p$, the character $\chi_q$ is 
		unramified and $\chi_q(q)=\Big( \frac{q}{p}\Big)$.
		\item 
		$\chi_p$ is ramified with conductor $p$ and the restriction
		$\chi|_{\Z_p^\times}$ factors through the unique quadratic
		character of $\mathbb{F}_p^\times$ with $\chi_p(p)=1$.
	\end{enumerate}
	By definition, we say that $\chi_p$ is tamely ramified. For $p=2$, there are three quadratic extensions ramified only at $2$, namely $\Q(\sqrt{-1}),\Q(\sqrt{2})$ and $\Q(\sqrt{-2})$ with the corresponding characters denoted by $\chi_{-1}, \chi_2$ and $\chi_{-2}$ respectively (the first one has conductor $2$ whereas the last two have conductors $3$).
	
	\begin{theorem}\label{q}
		Let $p$ be an odd prime and $q \neq p$. Then, $\varepsilon_q=\left(\frac{q}{p}\right)^{{\rm val}_q(a({\sym}^3(\pi)))}$.
	\end{theorem}
	
	\begin{proof}
		As $\chi_q$ is unramified for $q \neq p$, using \cite[Equ. 5.5.1]{deligne} we have that $\varepsilon(\sym^3(\pi_{q}) \otimes \chi_q, \phi) = \chi_q(q^{a({\rm sym}^3(\pi_{q}))+n(\phi)\cdot{\rm dim}(\sym^3(\pi_{q}))})\varepsilon(\sym^3(\pi_{q}), \phi )$. Now, since dim$(\sym^3(\pi_{q}))=4$ and $\chi_q(q)=\Big( \frac{q}{p}\Big)$, we obtain the desired value of $\varepsilon_q$.
	\end{proof}

	\section{Properties of $\sym^3$ transfers depending on local representations}
	\label{localprop}
	In this section, we compute the two properties under investigation for $\sym^3$ transfers depending on the types of local automorphic representations $\pi_p$. 
	\subsection{Principal Series Representations}
	In this case, if $f$ is a $p$-minimal newform then the local represtation of $f$ at $p$ is of the form $\pi_{p} \cong \pi(\mu_1, \mu_2)$, where $\mu_1$ is unramified,	$\mu_1(p)=\frac{a_p}{p^{(k-1)/2}}$, $\mu_1 \mu_2=\omega_p$ has conductor $p^{N_p}$ with $N_p \geq 1$ \cite[Prop. 2.8]{LW}. Here, $\omega_p$ is the $p$-part of the central character of $\pi_f$. The Laglands parameter $(\rho, N')$ of the symmetric cube transfer $\sym^3(\pi_p)$ is given by 
	\begin{equation} \label{prinsym3} \rho = \mu_1^3 \oplus \mu_1\omega_p \oplus \mu_1^{-1}\omega_p^2 \oplus \mu_1^{-3} \omega_p^3 \text{ with } N'=0 \end{equation} whereas $(\rho \otimes \chi_p, N')$ is the Laglands parameter of $\sym^3(\pi_p \otimes \chi_p)$ with 
	\begin{equation} \label{prinsym33}
	\rho \otimes \chi_p = \mu_1^3 \chi_p \oplus \mu_1\omega_p \chi_p \oplus \mu_1^{-1}\omega_p^2\chi_p \oplus \mu_1^{-3}\omega_p^3\chi_p.
	\end{equation}

	The proposition below determines the conductor of the symmetric cube transfer of $\pi_p$.
	\begin{prop} \label{aprin}
		Let $f$ be a $p$-minimal newform such that $\pi_{f,p}$ is a principal series representation. Then for primes $p\geq 5$, we have
		\begin{eqnarray*}
			a({\rm sym}^3(\pi_p))=
			\begin{cases}
				3N_p \quad \quad \quad \quad \quad \, \, \, \, \text{ if } N_p>1 \text{ or } N_p=1 \text{ with } \circ(\omega_p|_{\Z_p^\times}) > 3\\
				3N_p-1 \quad \quad \quad \quad \text{ otherwise }
			\end{cases}.
		\end{eqnarray*}
		For primes $p=3$ and $p=2$, the conductors are given below.
		\begin{gather} \label{p=2,3}
		a({\rm sym}^3(\pi_3))=
		\left\{\begin{array}{lll}
		2N_3 & \text{ if } N_3=2, \circ(\omega_3|_{\Z_3^\times}) = 3  \\            
		3N_3-1 & \text{ otherwise},
		\end{array}\right.
		a({\rm sym}^3(\pi_2))=
		\left\{\begin{array}{lll}
		3N_2-1 & \text{ if } N_2>3 \\             
		2N_2 & \text{ otherwise.}
		\end{array}\right.
		\end{gather}
	\end{prop}
	
	\begin{proof}
		Since $f$ is $p$-minimal, we have $N_p=C_p=a(\omega_p) \geq 1$. Using Equ. \ref{prinsym3}, we have $a({\rm sym}^3(\pi_p))= a(\omega_p)+a(\omega_p^2)+a(\omega_p^3)$. For primes $p \geq 3$, we use Prop. \ref{chiq} to get the desired result.
		For the prime $p=2$, we use Prop. \ref{chi2} to calculate the value of $a({\rm sym}^3(\pi_2))$.
	\end{proof}

	We now determine the variance number $\varepsilon_p$ in the theorem below. Here we choose an additive character $\phi$ of conductor $-1$. Set $c:=\prod_{i=1}^3 c_i$, where $c_i$ is given by the relation $\omega_p^i(1+x)=\phi(c_ix) ~~\forall~ x\in \mathfrak{p}_{\Q_p}^r,  i\in \{1,2,3\}$, see Lemma \ref{relation between multiplicative and additive character}.
	\begin{theorem}\label{prin}
		Let $f$ be a $p$-minimal newform with $\pi_{f, p}$ of ramified principal series type.  If $\phi$ is an additive character of conductor $-1$, then we have the following:
		\begin{enumerate}
			\item
			Let $p \geq 5$. If $N_p>1$, then the number
			\begin{eqnarray*} 
				\varepsilon_p=
				\begin{cases}
					\chi_p(c)p^{\frac{3-3k}{2}}a_p^3, & \quad \text{if} \,\, p \equiv 1 \pmod{4}, \\
					i\chi_p(c)p^{\frac{3-3k}{2}}a_p^3, & \quad \text{if} \,\, p \equiv 3 \pmod{4},
				\end{cases}
			\end{eqnarray*}
			where $c$ has valuation$-3(N_p-1)$. If $N_p=1$, the value of $\varepsilon_p$ is given in Table \ref{tab:table2}.
			\item
			For $p=3$ with $N_3>1$, $\varepsilon_3$ is given in Table \ref{tab:table3}. If $N_3=1$, then $\varepsilon_3=3^{2-2k}a_3^4$.
			\item
			Suppose that $p=2$. If $N_2=2$ then $\varepsilon_2=\frac{a_2^8 2^{4-4k}}{\omega_2^4(2)}$  and if $N_2=3$ then $\varepsilon_2=-\frac{a_2^4 \omega_2^4(2)}{2^{2k}}$. For $N_2>4$, we have $\varepsilon_2=i\chi_{-1}(2c) 2^{2-3k} a_2^{6}$ where $c$ has valuation $-3N_2+4$. The case $N_2=1$ is not possible.
		\end{enumerate}
	\end{theorem}
	
	\begin{proof}
		Using Equations \ref{prinsym3} and \ref{prinsym33} together with \cite[Property $(\epsilon' 1), p. 143$]{Roh}, we have that
		\begin{eqnarray} \label{prin1}
		\varepsilon_p= \frac{\varepsilon(\mu_1^3\chi_p,\phi)\varepsilon(\mu_1\omega_p\chi_p,\phi)\varepsilon(\mu_1^{-1}\omega_p^2\chi_p,\phi)\varepsilon(\mu_1^{-3}\omega_p^3\chi_p,\phi) }{\varepsilon(\mu_1^3,\phi)\varepsilon(\mu_1\omega_p,\phi)\varepsilon(\mu_1^{-1}\omega_p^2,\phi)\varepsilon(\mu_1^{-3}\omega_p^3,\phi)},
		\end{eqnarray}
		where $\phi$ is as choosen above.  
		Now we calculate each local epsilon factors involved in $\varepsilon_p$ separately.
		
		(1) First assume that $p\geq5$ and $N_p >1$. By property $(\epsilon 2)$ and the fact that $\mu_1$ is unramified, we list the following properties:
		\begin{enumerate}
			\item 
			$\varepsilon(\mu_1^3\chi_p,\phi)= \mu_1^3(p)^{a(\chi_p)-1} \varepsilon(\chi_p, \phi)$ $= \varepsilon(\chi_p,\phi)$.
			
			\item 
			$\varepsilon(\mu_1^3, \phi)= \frac{1}{\mu_1^3(p)}$ as $\mu_1^3$ is unramified.
			
			\item 
			$\varepsilon(\mu_1 \omega_p \chi_p,\phi)= \mu_1(p)^{a(\omega_p\chi_p)-1}\varepsilon(\omega_p\chi_p,\phi)$ $=\mu_1(p)^{a(\omega_p)-1}\varepsilon(\omega_p\chi_p,\phi)$ as $a(\omega_p)>a(\chi_p)$. 
			
			\item 
			$\varepsilon(\mu_1\omega_p,\phi)= \mu_1(p)^{a(\omega_p)-1}\varepsilon(\omega_p,\phi)$.
			
			\item 
			$\varepsilon(\mu_1^{-1}\omega_p^2\chi_p,\phi)= \mu_1^{-1}(p)^{a(\omega_p^2\chi_p)-1}\varepsilon(\omega_p^2\chi_p,\phi)=\mu_1^{-1}(p)^{a(\omega_p)-1}\varepsilon(\omega_p^2\chi_p,\phi)$. This is because when $p\geq 5$, we have $a(\omega_p^2)=a(\omega_p) > a(\chi_p)$ and hence $a(\omega_p^2\chi_p)=a(\omega_p)$.
			
			\item 
			$\varepsilon(\mu_1^{-1}\omega_p^2,\phi)= \mu_1^{-1}(p)^{a(\omega_p^2)-1}\varepsilon(\omega_p^2,\phi)= \mu_1^{-1}(p)^{a(\omega_p)-1}\varepsilon(\omega_p^2,\phi)$.
			
			\item 
			$\varepsilon(\mu_1^{-3}\omega_p^3\chi_p,\phi)= \mu_1^{-3}(p)^{a(\omega_p^3\chi_p)-1}\varepsilon(\omega_p^3\chi_p,\phi)= \mu_1^{-3}(p)^{a(\omega_p)-1}\varepsilon(\omega_p^3\chi_p,\phi)$ as $a(\omega_p^3\chi_p)=a(\omega_p)$.
			
			\item $\varepsilon(\mu_1^{-3}\omega_p^3,\phi)=\mu_1^{-3}(p)^{a(\omega_p^3)-1}\varepsilon(\omega_p^3,\phi)=\mu_1^{-3}(p)^{a(\omega_p)-1}\varepsilon(\omega_p^3,\phi)$. 
			
		\end{enumerate} 
		
		\noindent
		Using the above relations, from \eqref{prin1} we have that
		
		\begin{eqnarray} \label{prin2}
		\varepsilon_p=\mu_1^3(p) \frac{\varepsilon(\chi_p,\phi)\varepsilon(\omega_p\chi_p,\phi)\varepsilon(\omega_p^2\chi_p,\phi)\varepsilon(\omega_p^3\chi_p,\phi) }{\varepsilon(\omega_p,\phi)\varepsilon(\omega_p^2,\phi)\varepsilon(\omega_p^3,\phi)}.
		\end{eqnarray}
		

		Since $a(\omega_p^i)=a(\omega_p), i\in \{1,2,3\}$, taking $\chi=\omega_p^i$ in Lemma \ref{relation between multiplicative and additive character}, we get
		$$\omega_p^i(1+x)=\phi(c_ix) ~~\forall~ x\in \mathfrak{p}_{\Q_p}^r,~ 2r \geq N_p,$$
		for some $c_i\in \Q_p^\times$ with valuation $-(a(\omega_p)+n(\phi))$. Now using Theorem \ref{epsilon factor while twisting}, we get
		
		\begin{eqnarray} \label{prin3}
		\varepsilon(\omega_p^i\chi_p,\phi)=\chi_p^{-1}(c_i)\varepsilon(\omega_p^i,\phi), i \in \{1,2,3\}.
		\end{eqnarray}
		Thus, putting the value of $\mu_1(p)$ and applying Lemma \ref{alpha}, we get that
		\begin{eqnarray*}
			\varepsilon_p &=&  \mu_1^3(p)\varepsilon(\chi_p,\phi)\prod_{i=1}^3\chi_p^{-1}(c_i)  \\
			&=& \chi_p(c) \mu_1^3(p) \varepsilon(\chi_p,\phi) \text{ where } c=c_1c_2c_3 \text{ has valuation } -3(N_p-1).\\
			&=& \begin{cases}
								\chi_p(c)p^{\frac{3-3k}{2}}a_p^3, & \quad \text{if} \,\, p \equiv 1 \pmod{4}, \\
								i\chi_p(c)p^{\frac{3-3k}{2}}a_p^3, & \quad \text{if} \,\, p \equiv 3 \pmod{4}.
							\end{cases}
		\end{eqnarray*}
		
		
		\smallskip 
		Next, we assume $N_p=1$ for an odd prime $p\geq5$ i.e. $a(\omega_p)=1$. We have from \eqref{prin1}, 
		\begin{eqnarray}\label{prin5}
		\begin{split}
		\varepsilon_p & = \frac{\mu_1(p)^{3+a(\omega_p\chi_p)-a(\omega_p^2\chi_p)-3a(\omega_p^3\chi_p)}}
		{\mu_1(p)^{a(\omega_p)-a(\omega_p^2)-3a(\omega_p^3)}}\times \frac{\varepsilon(\chi_p,\phi)\prod_{i=1}^{3}\varepsilon(\omega_p^i\chi_p, \phi)}{\prod_{i=1}^{3}\varepsilon(\omega_p^i,\phi)}
		\end{split}	 
		\end{eqnarray}
		Write $t_1=3+a(\omega_p\chi_p)-a(\omega_p^2\chi_p)-3a(\omega_p^3\chi_p)-a(\omega_p)+a(\omega_p^2)+3a(\omega_p^3)$. Next, we compute $a(\omega_p^i\chi_p^j), i \in \{1,2,3\}, j \in \{0,1\}$ and the values of $t_1$ depending upon the order of $\widetilde{\omega_p}:=\omega_p^{-1}|_{\mathbb{Z}_p^\times}$.

		\begin{table}[h!]
			\begin{center}
				\begin{tabular}{|c|c|c|c|c|c|c|c|}
					\hline
					$o(\widetilde{\omega_p})$ & $a(\omega_p)$ & $a(\omega_p^2)$ & $a(\omega_p^3)$ & $a(\omega_p\chi_p)$ & $a(\omega_p^2\chi_p)$ & $a(\omega_p^3\chi_p)$ & $t_1$\\
					\hline
					2 & 1 & 0 & 1 & 0 & 1 & 0 & 4\\
					3 & 1 & 1 & 0 & 1 & 1 & 1 & 0\\
					4 & 1 & 1 & 1 & 1 & 0 & 1 & 4\\
					6 & 1 & 1 & 1 & 1 & 1 & 0 & 6\\
					$5 \text{ or}>6$ & 1 & 1 & 1 & 1 & 1 & 1 & 3\\
					\hline
				\end{tabular}
				\vskip 1mm
				\caption{ }
				\label{tab:table1}
			\end{center}
		\end{table}	
		We now compute $\varepsilon(\omega_p^j, \phi)$ and $\varepsilon(\omega_p^j\chi_p, \phi)$.  By Remark \ref{R1}(2), 
		the characters $\widetilde{\omega_p}, \widetilde{\chi_p}=\chi_p|_{\mathbb{Z}_p^\times} \in  \widehat{\mathbb{F}_p^\times}=\langle \chi_1 \rangle$ for some $\chi_1$. Let $\widetilde{\omega_p}$ has order $m$ and write $\widetilde{\omega_p}= \chi_1^a$. Then $\widetilde{\omega_p}^j\widetilde{\chi_p}=\chi_1^{ja+\frac{p-1}{2}}$. Now by definition and Equ. \ref{GKcoro}, for $j=1, 2, 3$, we have that
		\begin{equation} \label{prin8}
		\varepsilon(\omega_p^j, \phi) = p^{-\frac{1}{2}} G_1(\chi_1^{ja}) =p^{-\frac{1}{2}}(-p)^{\frac{j}{m}}\Gamma_p \Big(\frac{j}{m} \Big),
		\end{equation}
		\begin{equation} \label{prin9}
		\varepsilon(\omega_p^j\chi_p, \phi) = p^{-\frac{1}{2}} G_1(\chi_1^{ja+\frac{p-1}{2}}) =p^{-\frac{1}{2}}(-p)^{\frac{j}{m}+\frac{1}{2}}\Gamma_p \big(\frac{j}{m}+\frac{1}{2} \big).
		\end{equation}
		We will use $\varepsilon(\alpha, \phi)=\frac{1}{\alpha(p)}$ if $\alpha$ is unramified and Lemma \ref{alpha} to compute it when $a(\alpha)=1$. 
		
		Assume that  $\circ(\widetilde{\omega_p})=2$. As $\chi_p(p)=1$ and $t_1=4$, we obtain from \eqref{prin5} that $\varepsilon_p=\mu_1^4(p)=\frac{p^{2-2k}a_p^4}{\omega_p^2(p)}$.
		Now, assume that $\circ(\widetilde{\omega_p})=3$. In this case, $\varepsilon(\omega_p^3,\phi)= \frac{1}{\omega_p^3(p)}$. Using Lemma \ref{alpha} we compute $\varepsilon(\omega_p^3\chi_p,\phi)$. For the epsilon factors of $\omega_p, \omega_p^2$ we use \eqref{prin8}, and we use \eqref{prin9} for that of $\omega_p\chi_p, \omega_p^2\chi_p$. 
		As $t_1=0$, from \eqref{prin5} we get that
		\begin{eqnarray*} 
			\varepsilon_p=
			\begin{cases}
				-p\omega_p^3(p)c_p & \quad \text{if} \,\, p \equiv 1 \pmod{4} \\
				p\omega_p^3(p)c_p & \quad \text{if} \,\, p \equiv 3 \pmod{4},
			\end{cases}
		\end{eqnarray*}
		where $c_p=\frac{\Gamma_p(\frac{1}{m}+\frac{1}{2}) \Gamma_p(\frac{2}{m}+\frac{1}{2})}{\Gamma_p(\frac{1}{m})\Gamma_p(\frac{2}{m})}$ with $m=3$. 
		Next, assume $\circ(\widetilde{\omega_p})=4$. Proceeding as before, $\varepsilon_p=-p \times \frac{\mu_1^4(p)}{\omega_p^2(p)} \frac{\Gamma_p(\frac{1}{m}+\frac{1}{2}) \Gamma_p(\frac{3}{m}+\frac{1}{2})}{\Gamma_p(\frac{1}{m})\Gamma_p(\frac{3}{m})}$ with $m=4$. Using the property $\frac{\Gamma_p(1+x)}{\Gamma_p(x)}=-x$ if $x \in \Z_p^\times$, we obtain that $\varepsilon_p=\frac{p^{3-2k}}{4\omega_p^2(p)}a_p^4$. By the similar calculations for other cases of $\circ(\widetilde{\omega_p})$, we get the values of $\varepsilon_p$ which we write in Table \ref{tab:table2}:
		\begin{table}[h!]
			\begin{center}
				\begin{tabular}{|c|c|c|}
					\hline
					$\circ(\widetilde{\omega_p})$ & $\varepsilon_p$ & $c_p$ \\
					\hline
					\multirow{2}{*}{$2$} &
					\multirow{2}{*}{$\frac{p^{2-2k}a_p^4}{\omega_p^2(p)}$} & 
					\multirow{2}{*}{} \\ & {} & {}\\
					\hline
					\multirow{2}{*}{3} & $-p\omega_p^3(p)c_p  \quad \text{if}\,\, p \equiv 1 \pmod{4}$ & $c_p=\frac{\Gamma_p(\frac{5}{6}) \Gamma_p(\frac{7}{6})}{\Gamma_p(\frac{1}{3}) \Gamma_p(\frac{2}{3})}$\\ & $p\omega_p^3(p)c_p \quad \text{if} \,\, p \equiv 3 \pmod{4}$ & \\
					\hline
					\multirow{2}{*}{$4$} &
					\multirow{2}{*}{$\frac{p^{3-2k}a_p^4}{4\omega_p^2(p)}$} & 
					\multirow{2}{*}{} \\ & {} & {}\\
					\hline
					\multirow{2}{*}{$6$} &
					\multirow{2}{*}{$\frac{-p^{4-3k}a_p^6c_p}{\omega_p^3(p)}$} & 
					\multirow{2}{*}{$c_p=\frac{\Gamma_p(\frac{2}{3}) \Gamma_p(\frac{5}{6})}{\Gamma_p(\frac{1}{3}) \Gamma_p(\frac{1}{6})}$} \\ & {} & {}\\
					\hline
					\multirow{2}{*}{$5 \text{ or}>6$} & $a_p^3p^{\frac{3-3k}{2}}c_p \quad \text{if} \,\, p \equiv 1 \pmod{4}$ &
					$c_p=(-p)^{\frac{3}{2}}\frac{\prod_{j=1}^{3}\Gamma_p(\frac{j}{m}+\frac{1}{2})}{\prod_{j=1}^{3}\Gamma_p(\frac{j}{m})}$\\ &
					$ia_p^3p^{\frac{3-3k}{2}}c_p \quad \text{if} \,\, p \equiv 3 \pmod{4}$  & \\
					\hline
				\end{tabular}
				\vskip 1mm
				\caption{ }
				\label{tab:table2}
			\end{center}
		\end{table}	
	
		(2) Let $p=3$ with $N_p>1$. Note that $a(\omega_p) \neq a(\omega_p^3)$ by Prop. \ref{chiq}. As all the quotients in \eqref{prin1} are the same as in the case of primes $p \geq 5$ with $N_p>1$ expect $\frac{\varepsilon(\mu_1^{-3}\omega_p^3\chi_p,\phi)}{\varepsilon(\mu_1^{-3}\omega_p^3,\phi)}$, $\varepsilon_3$ is determined by the quantity in \eqref{prin2} with an extra factor $\mu_1^{-3}(p)^{a(\omega_p^3\chi_p)-a(\omega_p^3)}$. Hence, using \eqref{prin3} we obtain that
		\begin{eqnarray} \label{prin6}
		\varepsilon_p=\mu_1^3(p) \varepsilon(\chi_p,\phi)\chi_p(c_1c_2) \times \mu_1^{-3}(p)^{a(\omega_p^3\chi_p)-a(\omega_p^3)} \frac{\varepsilon(\omega_p^3\chi_p,\phi)}{\varepsilon(\omega_p^3,\phi)}.
		\end{eqnarray}
		
		We now calculate $\varepsilon(\omega_p^3\chi_p,\phi)$ and $\varepsilon(\omega_p^3,\phi)$. As $N_p \geq 2$, we note that $\circ(\widetilde{\omega_p})\neq 2$. If $\circ(\widetilde{\omega_p})=3$, then $\omega_p^3$ is unramified. Thus, $\varepsilon(\omega_p^3\chi_p,\phi)=\varepsilon(\chi_p, \phi)$ and $\varepsilon(\omega_p^3,\phi)=1/\omega_p^3(p)$. If $\circ(\widetilde{\omega_p})=6$, then $\omega_p^3$ is quadratic tamely ramified and $\omega_p^3\chi_p$ is trivial on $\Z_p^\times$. Thus, $\varepsilon(\omega_p^3\chi_p,\phi) =1/\omega_p^3(p)$ as $\chi_p(p)=1$, and $\varepsilon(\omega_p^3,\phi)$ is computed by Lemma \ref{alpha}. If $\circ(\widetilde{\omega_p})>6$ then $a(\omega_p^3)>1$ by Prop. \ref{chiq}, so $\frac{\varepsilon(\omega_p^3\chi_p,\phi)}{\varepsilon(\omega_p^3,\phi)}=\chi_p(c_3)$ using Theorem \ref{epsilon factor while twisting}. Hence, if $N_3>1$ then $\varepsilon_3$ is given in Table \ref{tab:table3}.
		
		\pagebreak
		
		\begin{table}[h!]
			\begin{center}
				\begin{tabular}{|c|c|c|c|}
					\hline
					$\circ(\widetilde{\omega_3})$ & $\varepsilon_3$ & $\circ(\widetilde{\omega_3})$ & $\varepsilon_3$  \\
					\hline
					\multirow{2}{*}{$3$} &
					\multirow{2}{*}{$-\chi_3(c_1c_2) \omega_3^3(3)$} &
					\multirow{2}{*}{$6$} &
					\multirow{2}{*}{$\frac{3^{3-3k}  \chi_3(c_1c_2) a_3^6}{\omega_3^3(3)}$} \\ & {} & {} & {} \\
					\hline
					\multirow{2}{*}{$>6$} &
					\multirow{2}{*}{$i\chi_3(c)3^{\frac{3-3k}{2}}a_3^3$} &
					\multirow{2}{*}{} &
					\multirow{2}{*}{} \\ & {} & {} & {} \\
					\hline
				\end{tabular}
				\vskip 1mm
				\caption{ }
				\label{tab:table3}
			\end{center}
		\end{table}			
		
		Next, we consider the case $p=3$ with $N_p=1$. Note that Equ. \ref{prin5} determines the variation number $\varepsilon_p$. As $p=3$ and $a(\omega_p)=1$, we must have that $\circ(\widetilde{\omega_p})=2$. Thus, the same computation as in the case of primes $p \geq 5$ produces that $\varepsilon_3= \frac{3^{2-2k}a_3^4}{\omega_3^2(3)}$.
		
		\smallskip 
		
		(3) Let $p=2$ with $N_2>4$. As $\chi_p=\chi_{-1}$ has conductor $2$ and $a(\omega_p) \neq a(\omega_p^2)$ by Lemma \ref{chi2}, using Equ. \ref{prin1} we deduce that 
		$\varepsilon_p=\mu_1^6(p) \varepsilon(\chi_p,\phi)\chi_p(c_1c_3) \times \mu_1^{-1}(p)^{a(\omega_p^2\chi_p)-a(\omega_p^2)} \frac{\varepsilon(\omega_p^2\chi_p,\phi)}{\varepsilon(\omega_p^2,\phi)}.$
		Note that $\circ(\widetilde{\omega_p}) > 4$. 
		Therefore, $a(\omega_p^2) \geq 4$. Thus, we can use Theorem \ref{epsilon factor while twisting} and Equ. \ref{prin3} to compute the ratio $\frac{\varepsilon(\omega_p^2\chi_p,\phi)}{\varepsilon(\omega_p^2,\phi)}$ that equals to $\chi_p(c_3)$. Hence, by Lemma \ref{alpha}, $\varepsilon_2=i\chi_{-1}(2c) 2^{2-3k} a_2^{6}$ where $c$ has valuation $-3N_2+4$. 
		
		Now, assume $N_2=3$. As $\circ(\widetilde{\omega_p})=2$, $\omega_p^2$ is unramified. Also, $a(\omega_p\chi_p)=a(\omega_p^3 p\chi_p)=3$. Using the epsilon factors of an unramified character or twisted by an unramified character together with Lemma \ref{alpha}, from Equ. \ref{prin1} we obtain that $\varepsilon_2=-\frac{a_2^4 \omega_2^4(2)}{2^{2k}}$.
		Next, assume that $N_2=2$. As $\circ(\widetilde{\omega_p})=2$, the characters $\omega_p^2$, $\omega_p\chi_p$ and $\omega_p^3\chi_p$ are unramified. Proceeding as the previous case, we get that $\varepsilon_2=\frac{a_2^8 2^{4-4k}}{\omega_2^4(2)}$.
		As $\Z_2^\times=1+2\Z_2$, there is no character of conductor $1$ of $\Q_2^\times$. Hence the case $N_2=1$ does not occur.
	\end{proof}

	\subsection{Special Representations}  \label{sp}
	For a $p$-minimal newform $f$, if the local represtation of $f$ at $p$ is of {\it special type} then we have $N_p=1$ and $\pi_{f,p} = \mu \otimes \mathrm{St}_2$. Note that  both $\omega_p$ and $\mu$ are unramified with $\mu(p)=\frac{a_p}{p^{(k-2)/2}}$ \cite[Prop. 2.8]{LW}. The Langlands parameter $(\phi, N)$ of $\mu \otimes \mathrm{St}_2$ is given in Equ. \ref{spl}. Now, using Equ. \ref{symcube} we obtain that
	\begin{eqnarray*} 
		{\rm sym}^3(\phi)(w)= \begin{bmatrix}
			\mu^3(w)|w|^{\frac{3}{2}} & & & \\ & \mu^3(w)|w|^{\frac{1}{2}} & & \\ & & \mu^3(w)|w|^{-\frac{1}{2}} & \\ & & & \mu^3(w)|w|^{-\frac{3}{2}}
		\end{bmatrix}, w\in W(\Q_p).
	\end{eqnarray*}  
	
	Let $N'$ be the associated nilpotent endomorphism for the symmetric cube transfer of $\pi_{f,p}$. Then to calculate $N'$, we consider the following commutative diagram:
	\[\begin{tikzcd}
	{\rm GL}_2(\mathbb{C}) \arrow{r}{{\rm sym}^3}  & {\rm GL}_4(\mathbb{C}) \arrow{d}{d({\rm sym}^3)|_{t=0}}\\ {\mathfrak{gl}_2({\mathbb{C}})} \arrow{u}{\exp} \arrow{r} & {\mathfrak{gl}_4({\mathbb{C}})}
	\end{tikzcd}\] 
	where $\mathfrak{gl}_n(\mathbb{C})$ is the Lie algebra of ${\rm GL}_n(\mathbb{C})$. Since $\exp(tN)= \begin{bmatrix}
	1 & t \\ 0 & 1
	\end{bmatrix}$ and ${\rm sym}^3(\exp(tN))=\begin{bmatrix}
	1 & t & t^2 & t^3 \\ 0 & 1 & 2t & 3t^2 \\ 0 & 0 & 1 & 3t \\ 0 & 0 & 0 & 1
	\end{bmatrix}$ $\in {\rm GL}_4(\mathbb{C})$, we have 
	$A=d({\rm sym}^3(\exp(tN)))|_{t=0}= \begin{bmatrix}
	0 & 1 & 0 & 0 \\ 0 & 0 & 2 & 0 \\ 0 & 0 & 0 & 3 \\ 0 & 0 & 0 & 0
	\end{bmatrix}.$
	Taking $B= \begin{bmatrix}
	1 & 0 & 0 & 0 \\ 0 & 1 & 0 & 0 \\ 0 & 0 & 2 & 0 \\ 0 & 0 & 0 & 6
	\end{bmatrix}$, we have $BAB^{-1}=N'=\begin{bmatrix}
	0 & 1 & 0 & 0 \\ 0 & 0 & 1 & 0 \\ 0 & 0 & 0 & 1 \\ 0 & 0 & 0 & 0
	\end{bmatrix}.$
	Thus we have shown the following.
	\begin{prop}
		Let $\pi_{f,p}$ be a special representation attached to a $p$-minimal form $f$ with the Langlands parameter $(\mu \otimes \mathrm{St}_2, N)$, then $(\mu^3 \otimes {\rm St}_4, N')$ is the Langlands parameter of ${\rm sym}^3(\pi_{f,p})$.
	\end{prop}
	
	\medskip 
	
	\subsection{Computation of $\varepsilon$-factor:} For a character $\mu$ of $\Q_p^\times$, let us compute the $\varepsilon$-factor of the representation $\rho'=(\rho, N')$, where $\rho=\mu \otimes \rm{St}_4$, that acts on the space $V$. Let $\Phi \in W(\Q_p)$ be an inverse Frobenius element and $I$ be the inertia subgroup of $W(\Q_p)$. Consider the following subspaces of $V$:
	\[V^I=\{ v \in V : \rho(g) v =v \text{ for all } g \in I\},\]
	\[V_{N'} =\ker(N') \text{ and } V_{N'}^I=V^I\cap V_{N'}.\]
	Using \cite[\S 10]{Roh}, we have $a(\rho')=a(\rho)+\dim V^I-\dim V_{N'}^I$. From this we deduce that \cite[\S 10, Prop.]{Roh}
	\begin{prop} \label{aspecial} $a(\sym^3(\pi_{f,p}))=3$. \end{prop}
	With the notations as above, the $\varepsilon$-factor of $\rho'$ is given as follows  \cite[Equ.(2.49)]{BR}: 
	\begin{eqnarray}  \label{specialepsilon}
	\varepsilon(s, \rho', \phi) = \varepsilon(s, \rho, \phi) \det (-\rho(\Phi) p^{-s} | V^I/V^I_{N'}).
	\end{eqnarray}
	\begin{prop} \label{specialprop}
		We have
		\begin{eqnarray*} 
			\varepsilon(s, \rho', \phi)=
			\begin{cases}
				\varepsilon(s, \rho, \phi), & \quad \text{if} \,\, \mu \text{ is ramified}, \\
				-\mu^3(p)p^{-3s+\frac{1}{2}} \varepsilon(s, \rho, \phi) & \quad \text{if} \,\, \mu \text{ is unramified}.
			\end{cases}
		\end{eqnarray*}
	\end{prop}
	\begin{proof}
		If $W$ is the space of $\mu$, then $V=W \otimes \C^4$ is the space of $\pi$. Letting $\{e_0, e_1, e_2, e_3\}$ be the standard basis of $\C^4$, we find that $\rm{ker} (N')=\C e_1$. Thus, we obtain that 
		\begin{eqnarray} \label{spaces}
		V^I=W^I \otimes \C^4 \quad \text{and} \quad V_{N'}^I=W^I \otimes \C e_1.
		\end{eqnarray}
		
		Let us now find $\varepsilon(s, \rho', \phi)$. First we assume that $\mu$ is ramified. In this case, we deduce that $W^I=\{0\}$ and thus we have $V^I=\{0\}$. Consequently, it follows that 
		\begin{eqnarray*}
			\varepsilon(s, \rho', \phi)=\varepsilon(s, \rho, \phi).
		\end{eqnarray*}
		Next, we assume that $\mu$ is unramified. Recall that $\rho$ is a direct sum of the unramified characters $\mu |.|^{\frac{i-j}{2}}$ for $0 \leq j \leq 3$ such that $i+j=3$. In this case, $W=W^I$ is $1$-dimensional. As a result, from \eqref{spaces} it follows that $V^I/V^I_{N'} = \oplus_{j \ne 1} (W\otimes e_j)$ with $-\rho(\Phi)$ acting by $-\mu(p)|p|^{\frac{i-j}{2}}$ such that $i+j=3$. Therefore, $\det (-\rho(\Phi) p^{-s} | V^I/V^I_{N'})=-\mu^3(p)p^{-3s}|p|^{-1/2}$. So using \eqref{specialepsilon} we deduce that
		\begin{eqnarray*}
			\varepsilon(s, \rho', \phi)=-\mu^3(p)p^{-3s+\frac{1}{2}} \varepsilon(s, \rho, \phi).
		\end{eqnarray*}
		This completes the proof of the proposition.
	\end{proof}
	
	Note that $\varepsilon(\frac{1}{2}, \rho, \phi)=\varepsilon(\rho, \phi)$, explained in \cite[\S 3]{bm}. The theorem below determines the variation number $\varepsilon_p$ when the local automorphic representation of $f$ at $p$ is of special type.
	\begin{theorem}
		If $f$ is a $p$-minimal newform with $\pi_{f, p}$ a special representation, then the variation number $\varepsilon_p$ is given by
		\begin{equation*} 
		\varepsilon_p= -p^{\frac{8-3k}{2}}a_p^3.
		\end{equation*}
		For $p=2$, we have $\varepsilon_2= -2^{\frac{24-15k}{2}} a_2^{15}$.
	\end{theorem}

	\begin{proof}
		Note that the Langlands parameter of ${\rm \sym}^3(\pi_{f, p})$ is $\rho'=(\rho, N')$ where $\rho=\mu^3 \otimes {\rm St}_4$ and $N'$ is given as above. Then $(\rho\otimes \chi_p, N')$ is the Laglands parameter of ${\rm \sym}^3(\pi_{f, p} \otimes \chi_p)$. As $\rho$ is the direct sum of the unramified characters $\mu^3 |.|^{\frac{i-j}{2}}$ for $0 \leq j \leq 3$ with $i+j=3$, by the properties of epsilon factors we obtain that $$\varepsilon(\rho, \phi)={\displaystyle \prod_{0\leq j \leq 3, i+j=3}} \varepsilon(\mu^3 |.|^{\frac{i-j}{2}}, \phi).$$
		Similarly $\varepsilon(\rho \otimes \chi_p, \phi)$ can be decomposed. As $\mu^3 |.|^{\frac{i-j}{2}}$ is unramified, by the property of epsilon factors for an unramified chracter we obtain that $\varepsilon(\mu^3 |.|^{\frac{i-j}{2}}, \phi)=(\mu^3 |.|^{\frac{i-j}{2}})(c)=\frac{p^{\frac{i-j}{2}}}{\mu^3(p)}$ as $c=1/p$. On the other hand, by $(\epsilon 2)$, we have $\varepsilon(\mu^3 |.|^{\frac{i-j}{2}} \chi_p, \phi)=\varepsilon(\chi_p, \phi)$. Combining all these observations together with Prop. \ref{specialprop} and Lemma \ref{alpha}, we obtain that
		\begin{eqnarray*}
			\varepsilon_p &=&\frac{\varepsilon(\rho' \otimes \chi_p, \phi)}{\varepsilon(\rho', \phi)} =\frac{1}{-\{\mu^3(p)\}^3 p^{-\frac{3}{2}+\frac{1}{2}}} ~~  {\displaystyle \prod_{0\leq j \leq 3, i+j=3}} ~~ \frac{\varepsilon(\mu^3 |.|^{\frac{i-j}{2}} \chi_p, \phi)}{\varepsilon(\mu^3 |.|^{\frac{i-j}{2}}, \phi) }  \\
			&=& -\frac{1}{\mu^{9}(p)p^{-1}} ~~  {\displaystyle \prod_{0\leq j \leq 3, i+j=3}}  ~~ \frac{\mu^3(p) \varepsilon(\chi_p, \phi)}{p^{\frac{i-j}{2}}}   \\
			&=& -\frac{1}{\mu^{9}(p)p^{-1}}  \times \mu^{12}(p)  \quad [\because \varepsilon(\chi_p, \phi)^4=1] \\
			&=& -p\mu^3(p) = -p^{\frac{8-3k}{2}}a_p^3.
		\end{eqnarray*}
	
		For $p=2$, as $\chi_p=\chi_{-1}$ has conductor $2$, we have $\varepsilon(\mu^3 |.|^{\frac{i-j}{2}} \chi_p, \phi)=\mu^3(p) |p|^{\frac{i-j}{2}} \varepsilon(\chi_p, \phi)$. By Lemma \ref{alpha}, $\varepsilon(\chi_{-1}, \phi)=\frac{i \chi_{-1}(2)}{2}$. Then proceeding as above, we get that $\varepsilon_2= -2^{\frac{24-15k}{2}} a_2^{15}$.
	\end{proof}

	\subsection{Supercuspidal representations} 
	
	\subsection{The case $p$ odd:}
	For an odd prime $p$, the local representation attached to $f$ at $p$ is induced from a character $\varkappa$ of an index two subgroup $W(K)$ of $W(\Q_p)$, that is, 
	\begin{equation}  \label{oddsupercuspidal}
	\pi_p:=\pi_{f, p}=\Ind_{W(K)}^{W(\Q_p)} (\varkappa).
	\end{equation}
	Here, $K$ is a quadratic extension of $\Q_p$. 
	As the local representation of $f$ at an odd prime $p$ is always irreducible, we have that $\varkappa \neq \varkappa^{\sigma}$, where $\sigma$ is a non-trivial element in $W(\Q_p) \setminus W(K)$. Note that $\pi_p$ has the matrix form as described in \eqref{LLPsup} with respect to a suitable basis. Also, using Equ. \ref{indconductor} the conductor of $\pi_p$ has the following description.
	\begin{equation}\label{conductorofsupercuspidalrep}
	N_p=a(\pi_p) = \begin{cases}
	2a(\varkappa) \quad \quad \, K/\mathbb{Q}_p ~~ \text{ unramified}\\
	1+a(\varkappa) \quad K/\mathbb{Q}_p ~~ \text{ ramified}
	\end{cases}.
	\end{equation}

	\begin{remark} \label{centralcharacter}
		The central character of $\pi_p$ is $\varkappa|_{\Q_p^\times} \cdot \omega_{K/\Q_p}$, where $\omega_{K/\Q_p}$ is the quadratic character of $\Q_p^\times$ associated to $K/\Q_p$ such that $\omega_{K/\Q_p} ({N_{K/\Q_p}(K^\times)})$ $= 1$. Hence we have  $\varkappa|_{\Q_p^\times} \cdot \omega_{K/\Q_p}=\epsilon_p^{-1}$ \cite[\S 2]{LW}. Now, evaluating $\varkappa$ at $N_{K/\Q_p}(x)$ for $x \in K^\times$, we deduce that  \begin{equation} \label{central}  \varkappa^\sigma=\varkappa^{-1} \epsilon_p' \text{ on } K^\times \text{ where } \epsilon_p'=\epsilon_p \circ N_{K/\Q_p}. \end{equation}
	\end{remark}
	
	\begin{lemma} \label{tunnel} \cite[Lemma $1.8$]{Tunnel}
		Let $E/K$ be a quadratic separable extension with
		residue degree $f$. 
		If $\eta \in \widehat{K^\times}$,
		then $f a(\eta \circ N_{E/K})=a(\eta)+a(\eta \omega_{E/K})-a(\omega_{E/K})$.
		If $\psi$ is a non-trivial additive character of $K$, then
		$n(\psi \circ \mathrm{Tr}_{E/K})=(2/f)n(\psi)+d(E/K)$.
		Here, $\omega_{E|K}$ denotes the non-trivial character of $K^\times$
		with 
		kernel $N_{E/K} (E^\times)$ 
		and $d(E/K)$ is the valuation of the discriminant of $E/K$.
	\end{lemma}
	
	\smallskip 
	
	Applying the symmetric cube lifting map \eqref{symcube} to the Laglands parameter of $\pi_p$ (cf. \eqref{LLPsup}) we have
	$$ \text{sym}^3(\varkappa)(x) = \begin{bmatrix} \varkappa^3(x) & & & \\ & \varkappa^2(x)\varkappa^\sigma(x) & & \\ & & \varkappa(x)(\varkappa^\sigma)^2(x) & \\ & & &(\varkappa^\sigma)^3(x) \end{bmatrix} \text{ } \text{ and}$$
	$$ \text{sym}^3(\varkappa)(\sigma) = \begin{bmatrix} & & & 1 \\ & & \varkappa^\sigma(\sigma^2) & \\ & (\varkappa^\sigma)^2(\sigma^2) & & \\ (\varkappa^\sigma)^3(\sigma^2) & & & \end{bmatrix} \text{ since }\varkappa^\sigma(\sigma^2)= \varkappa(\sigma^2).$$\medskip
	
	\noindent Thus, we deduce that $\text{sym}^3(\pi_p)$ is a sum of two $2$-dimensional representations i.e.,
	\begin{eqnarray}  \label{eq1}
	\text{sym}^3(\pi_p)= {\rm Ind}^{W(\Q_p)}_{W(K)}(\varkappa^3) \oplus {\rm Ind}^{W(\Q_p)}_{W(K)} (\varkappa^2\varkappa^\sigma).
	\end{eqnarray}
	As $\pi_p$ is irreducible $\Rightarrow \varkappa \neq \varkappa^{\sigma}$, we deduce that ${\rm Ind}^{W(\Q_p)}_{W(K)} (\varkappa^2\varkappa^\sigma)$ is always irreducible. Using Equ. \ref{central}, we get that $\varkappa^2\varkappa^\sigma = \varkappa \epsilon_p'$. Therefore, 
	\begin{eqnarray}  \label{eq2}
	\mathrm{sym}^3(\pi_p)= {\rm Ind}^{W(\Q_p)}_{W(K)}(\varkappa^3) \oplus {\rm Ind}^{W(\Q_p)}_{W(K)} (\varkappa \epsilon_p').
	\end{eqnarray}


	\subsubsection{Type of symmetric cube transfer of a supercuspidal representation} \label{symtype}
	Note that ${\rm Ind}^{W(\Q_p)}_{W(K)} (\varkappa_1)$ $\cong {\rm Ind}^{W(\Q_p)}_{W(K)} (\varkappa_2)\iff \varkappa_1=\varkappa_2$ or $\varkappa_1=\varkappa_2^\sigma$. Depending upon whether ${\rm Ind}^{W(\Q_p)}_{W(K)} (\varkappa^3)$ is irreducible or reducible, we have the following three types of the symmetric cube transfers of supercuspidal representations associated to $f$ at an odd prime:
	\begin{itemize}
		\item 
		{\bf (Type I):} ${\rm Ind}^{W(\Q_p)}_{W(K)} (\varkappa^3)$ is irreducible and it is isomorphic to ${\rm Ind}^{W(\Q_p)}_{W(K)} (\varkappa\epsilon_p')$. \smallskip 
		\item 
		{\bf (Type II):} ${\rm Ind}^{W(\Q_p)}_{W(K)} (\varkappa^3)$ is irreducible and it is not isomorphic to ${\rm Ind}^{W(\Q_p)}_{W(K)} (\varkappa\epsilon_p')$.  \smallskip
		\item 
		{\bf (Type III):} ${\rm Ind}^{W(\Q_p)}_{W(K)} (\varkappa^3)$ is reducible. \smallskip 
	\end{itemize}
	Note that ${\rm Ind}^{W(\Q_p)}_{W(K)} (\varkappa^3)$ is irreducible when $\varkappa^3 \neq (\varkappa^3)^\sigma \iff \varkappa^6 \neq \epsilon_p'^3$ by Equ. \ref{central}. 
	On the other hand, ${\rm Ind}^{W(\Q_p)}_{W(K)} (\varkappa^3)$ is reducible,  when $\varkappa^3=(\varkappa^3)^\sigma \iff \varkappa^6 = \epsilon_p'^3$. In this case, ${\rm Ind}^{W(\Q_p)}_{W(K)} (\varkappa^3)=\varphi \oplus \varphi \omega_{K/\Q_p}$ for some character $\varphi$ of $\Q_p^\times$ such that $\varkappa^3=\varphi \circ N_{K/\Q_p}$. This implies that $\varkappa^3|_{\Q_p^\times} = \varphi^2$. The lemma below will characterize $a(\varphi)$.
	\begin{lemma} \label{phi}
		Let $\sym^3(\pi_{p})$ be a Type III representation. If $K/\Q_p$ is unramified, then $a(\varphi)=a(\varkappa^3)$. For $K/\Q_p$ ramified with $f$ $p$-minimal, $p \geq 5$ is not possible and when $p=3$ with $a(\varkappa^3)> 1$, we have
		\begin{equation} 
		a(\varphi) = 
		\begin{cases}
		a(\varkappa^3)+1, 
		\quad \quad \,\,\,\,  \text{if} \,\, \varkappa^3|_{\mathbb{Z}_p^\times}=1,  \\
		\frac{1}{2}\big(a(\varkappa^3)+1\big),
		\quad \text{if} \,\,\, \varkappa^3|_{\mathbb{Z}_p^\times}\neq1.
		\end{cases}
		\end{equation}
		If $\varkappa^3$ is unramified, then $a(\varphi)= 1$ or $0$ depending upon $\phi$ is quadratic or not respectively. The case $a(\varkappa^3)=1$ is not possible.
	\end{lemma}
	
	\begin{proof}
		We have $\varkappa^6=\epsilon_p'^3$ on $\mco_K^\times$ and $\varkappa^3=\varphi \circ N_{K/\Q_p}$. 
		
		If $K/\Q_p$ is unramified, then by Lemma \ref{tunnel} we get $a(\varkappa^3)=a(\varphi)$ as $\omega_{K/\Q_p}$ is unramified.
		
		Let $K/\Q_p$ be ramified. Note that if $f$ is $p$-minimal, then $a(\varkappa)\geq 2$ is even \cite[Prop. 5.4]{bm}. First, assume that $p \geq 5$. If $C_p = a(\epsilon_p) \leq 1$, then by Lemma \ref{tunnel}, $\epsilon_p'$ is either unramified or tamely ramified. Thus, $\varkappa^6$ is also so and hence $p \nmid \circ(\varkappa|_{\mco_K^\times})$. On the other hand, $a(\varkappa) \geq 2$ implies that $p \mid \circ(\varkappa|_{\mco_K^\times})$. Thus, we arrive at a contradiction. Now, if $C_p>1$ then $a(\epsilon_p'^3)=a(\epsilon_p')$ by Prop. \ref{chi3} and using Lemma \ref{tunnel}, $a(\epsilon_p')\geq 3$. Hence, $a(\varkappa^6) \geq 3 \Rightarrow a(\varkappa^3) \geq 3$. The same lemma then deduce that $a(\varphi) \geq 2$ and so $a(\varkappa^3)=2 a(\varphi)-1$. Hence, $a(\varkappa^3)$ is odd and $a(\varphi)=\frac{1}{2}(a(\varkappa^3)+1)$. Note that both $a(\varkappa)$ and $a(\varkappa^3)$ have the same parity when $p\geq 5$. As a result, we obtain a contradiction to the $p$-minimality of $f$ as $a(\varkappa)$ is even. Hence, we conclude that $\sym^3(\pi_p)$ cannot be of Type III for a prime $p \geq 5$.
		
		The $p=3$ case follows from Lemma \ref{tunnel}.
		

		
	\end{proof}

	%

	
	\begin{prop} \label{cp}
		Let $f$ be a $p$-minimal form with $\pi_{f, p}=\Ind_{W(K)}^{W(\Q_p)}\varkappa$. Then
		\begin{equation} \label{c_p}
		a(\epsilon_p)=C_p \leq  \begin{cases}
		a(\varkappa) \quad \text{ if } K/\Q_p \text{ is unramified},\\
		\frac{a(\varkappa)}{2} \quad \, \text{ if } K/\Q_p \text{ is ramified}.
		\end{cases}
		\end{equation} 
	\end{prop}
	\begin{proof}
		To prove the proposition we will use \cite[Prop. 2.8]{LW} and Equ.\ref{conductorofsupercuspidalrep}. If $K/\Q_p$ is unramified, then $C_p \leq \lfloor \frac{2a(\varkappa)}{2} \rfloor=a(\varkappa)$. For $K/\Q_p$ ramified, $C_p \leq \lfloor \frac{1+a(\varkappa)}{2} \rfloor=\frac{a(\varkappa)}{2}$ as $a(\varkappa)$ is even.
	\end{proof}

\begin{prop}\label{conductorofsym3}
	For $p \geq 3$, we have the following conductor relation:
	\begin{equation}
	a({\rm sym}^3(\pi_p))= 
	\begin{cases}
	2a(\varkappa^3)+2a(\varkappa \epsilon_p') \quad \quad \, \, \, \, \, \text{ if } K/\mathbb{Q}_p \text{ unramified } \\
	a(\varkappa^3)+a(\varkappa \epsilon_p')+2 \quad \quad \text{ if } K/\mathbb{Q}_p \text{ ramified }
	\end{cases}.
	\end{equation}	
\end{prop}

\begin{proof}
	Let ${\rm sym^3}(\pi_p)$ be of Type I or II. From Equ. \ref{eq2}, $a({\rm sym^3}(\pi_p))= a\big(\Ind_{W(K)}^{W(\mathbb{Q}_p)}(\varkappa^3)\big)+a\big(\Ind_{W(K)}^{W(\Q_p)}(\varkappa \epsilon_p')\big)$. Using Equ. \ref{indconductor}, we have the desired value for $a({\rm sym}^3(\pi_p))$. 
	
	Next assume that ${\rm sym^3}(\pi_p)$ be of Type III. Then $a({\rm sym^3}(\pi_p))= a(\varphi)+a(\varphi \omega_{K/\Q_p})+a\big(\Ind_{W(K)}^{W(\Q_p)}(\varkappa \epsilon_p')\big)$. As $\varkappa^3=\varphi \circ N_{K/\Q_p}$, we use Lemma \ref{tunnel} and Equ. \ref{indconductor} to get the desired equality. 
\end{proof}

\begin{prop} \label{prop7}
	Let $p \geq 5$ be a supercuspidal prime for a $p$-minimal form $f$ with $\pi_p=\Ind_{W(K)}^{W(\Q_p)} (\varkappa)$. If $K/\Q_p$ is unramified, then under the assumption $C_p\neq \frac{N_p}{2}$ we have 
	\begin{equation}
	a({\rm sym}^3(\pi_p))= 
	\begin{cases}
	N_p \quad \text{ if } N_p=2 \text{ with} \circ(\widetilde{\varkappa}) = 3  \\
	2N_p \quad  \text{otherwise } 
	\end{cases}.
	\end{equation}	
	For $K/\Q_p$ ramified, we have $a(\sym^3(\pi_p)) =2N_p$. 
\end{prop}

\begin{proof}
	As $f$ is $p$-minimal, $N_p$ is even if $K/\Q_p$ is unramified, otherwise $N_p$ is odd \cite[Prop. 5.4]{bm}. First assume that $K/\Q_p$ is unramified. In this case, both $\epsilon_p$ and $\epsilon_p'$ have the same conductor. Thus, by Prop. \ref{cp}, we deduce that $C_p=a(\epsilon_p') \leq a(\varkappa)=\frac{N_p}{2}$. Then the assumption $C_p \neq \frac{N_p}{2}$ implies that $a(\varkappa\epsilon_p')=a(\varkappa)$. By Prop. \ref{chi33}, we have $a(\varkappa^3)=0$ if $N_p=2$ with $\circ(\widetilde{\varkappa}) = 3$; otherwise $a(\varkappa^3)=a(\varkappa)$. Now, using Equ. \eqref{conductorofsupercuspidalrep} and Prop. \ref{conductorofsym3}, we complete the proof when $N_p$ is even.

	Next, assume that $K/\Q_p$ is ramified. In this case, $a(\varkappa) \geq 2$ is even and from Prop. \ref{cp}, $a(\epsilon_p) \leq \frac{a(\varkappa)}{2} \Rightarrow 2a(\epsilon_p)-1 <a(\varkappa)$. Now, using Lemma \ref{tunnel} we conclude that $a(\varkappa\epsilon_p')=a(\varkappa)$. As a result,  using Equ. \ref{conductorofsupercuspidalrep} and Prop. \ref{conductorofsym3}, $a(\sym^3(\pi_p))=2(a(\varkappa)+1)=2N_p$.
\end{proof}

If $p=3$ is a supercuspidal prime for $f$, then we define the number $e_\varkappa$ depending upon $K/\Q_3$ unramified or ramified respectively, where $f_\varkappa$ is defined in \eqref{fchi}:
\begin{gather} \label{echi}
e_\varkappa=
\left\{\begin{array}{lll}
2 & \text {if } N_3=2  \\   
0 & \text {if }\circ(\widetilde{\varkappa}) =3  \\              
2f_\varkappa & \text{ otherwise }
\end{array}\right.
\quad \text{and}  \quad	e_\varkappa=
\left\{\begin{array}{lll}
1 & \text{ if } \circ(\widetilde{\varkappa}) =3 \\             
f_\varkappa +1 & \text{ otherwise.}
\end{array}\right.
\end{gather}
Using Prop. \ref{conductorofsym3}, we obtain the following.
\begin{prop} \label{ap=3}
	If $f$ is a $3$-minimal form with $\pi_{f, 3}=\Ind_{W(K)}^{W(\Q_3)} \varkappa$ and $C_3 \neq \frac{N_3}{2}$ when $K/\Q_3$ is unramified, then $a(\sym^3(\pi_3))=N_3+e_\varkappa$.
\end{prop}

\begin{prop} \label{types}
	For a $p$-minimal form $f$, the types of $\sym^3(\pi_p)$ are given in Table \ref{tab:table6} depending upon the values $N_p$ and $C_p$.
\end{prop}
	
\begin{proof} Note that if $p$ is a supercuspidal prime for a $p$-minimal form $f$, then $\pi_p=\Ind_{W(K)}^{W(\Q_p)} (\varkappa)$ with $C_p < N_p \geq 2$. Also, $N_p=2a(\varkappa)\geq 2$ is even if $K/\Q_p$ is unramified; otherwise $N_p=1+a(\varkappa) \geq 3$ is odd \cite[Prop. 5.4]{bm}. We will determine the type of $\sym^3(\pi_p)$ depending upon the values of $N_p$ and $C_p$.
	
	

	{\it Possibility of Type I:}  If $\sym^3(\pi_p)$ is of Type I, then $\varkappa^4=\epsilon_p'^2$ on $\mco_K^\times$.  First, assume that $K/\Q_p$ is unramified. Let $C_p \leq 1$ and $N_p\geq 4$. We claim,  $\sym^3(\pi_p)$ cannot be of Type I. If $C_p=0$, then as $\epsilon_p'$ is unramified we obtain that $\varkappa^4=1$ on $\mco_K^\times$. This is possible only when $a(\varkappa)=1$ i.e. $N_p=2$ which is a contradiction.  Now, let $C_p=1$. Hence, $\epsilon_p'$ is tamely ramified and so its order divides $(p^2-1)$. Then the order of $\varkappa^4|_{\mco_K^\times}$ divides $\frac{p^2-1}{2}$. This proves that $a(\varkappa^4) \leq 1$. Now, as $p$ is odd we also have $a(\varkappa) \leq 1$  and so $N_p=2$ which is again a contradiction. 
	
	\smallskip 
	
	Next, assume $K/\Q_p$ is ramified. For $C_p \leq 1$, we have $a(\epsilon_p') \leq 1$. This implies that $a(\varkappa^4) \leq 1$ i.e. $p \nmid \circ(\tilde{\varkappa})$. On the other hand, since 
	$a(\varkappa)\geq 2$ we have $p \mid \circ(\tilde{\varkappa})$. So we arrive at a contradiction. Hence, we deduce that, if $C_p\leq 1$, ${\rm sym}^3(\pi_p)$ cannot be of Type I. The same conclusion holds if $C_p=2$ with $N_p=3$. This shows that ${\rm sym}^3(\pi_p)$ is of Type I only if $C_p\geq 2$ with $N_p\geq 5$.
	
	
	
	\smallskip 
	
	{\it Possibility of Type III:}  If $\sym^3(\pi_p)$ is of Type III, we have $\varkappa^6=\epsilon_p'^3$ on $\mco_K^\times$. First, we consider $K/\Q_p$ is unramified. Assume that $p\geq 5$. Let $C_p \leq 1$ with $N_p\geq 4$.  We show that under these assumptions, $\sym^3(\pi_p)$ cannot be of Type III. If $C_p=0$, then $\epsilon_p'$ is unramified and so $\varkappa^6 =1$ on $\mco_K^\times$ i.e. $p \nmid \circ(\varkappa|_{\mco_K^\times})$. As $N_p\geq 4$ implies that $a(\varkappa)\geq 2$, we have $p \mid \circ(\varkappa|_{\mco_K^\times})$ which is a contradiction. If $C_p=1$, then $a(\epsilon_p')=1$. It follows that $a(\varkappa^6) \leq 1 \Rightarrow p \nmid \circ(\varkappa|_{\mco_K^\times})$. Hence,  by similar argument as above, $\sym^3(\pi_p)$ cannot be of Type III. 
	
	\smallskip
	
	Now, let $p=3$. Assume that $C_3 \leq 1$ with $N_3=2$. Further aassume that  $\varkappa^2|_{\mco_K^\times}\neq 1$. We show that, $\sym^3(\pi_3)$ cannot be of Type III. For $C_3=0$, using Lemma \ref{tunnel}, we have $\varkappa^6=1$ on $\mco_K^\times$. Since $\varkappa^2|_{\mco_K^\times}\neq 1$, we have $a(\varkappa)\geq 2 \implies N_3 \geq 4$ which contradicts the fact that $N_3=2$. Similarly we get a contradiction if $C_3=1$. Indeed, $C_3=1$ implies that $\epsilon_3'|_{\mco_K^\times}$ is quadratic as $N_{K/\Q_3}(\mco_K^\times)=\Z_3^\times$ \cite[Corollary 1.2, p. 319]{neukirch}. 
	Therefore, as $\varkappa^6=\epsilon_3'^3$, we have $\varkappa^{12}=1$ on $\mco_K^\times$. This gives us that $a(\varkappa) \geq 2$ and so $N_3 \geq 4$. 
	
	
	
	\smallskip 
	
	Let $K/\Q_p$ ramified. Then for $p \geq 5$, $\sym^3(\pi_p)$ cannot be of Type III by Lemma \ref{phi}. 
	

	{\small	
		\begin{table}[htbp]
			\centering
			\begin{tabular}{ccccc}
				\midrule[0.4mm]
				Type \quad\quad & $C_p$ \quad\quad & $N_p$\quad\quad  & Condition & Possibility \\
				\midrule[0.4mm]
				Type I, Unramified \quad\quad & $\leq 1$ \quad \quad & $2$ \quad\quad & &  \checkmark \\
				\quad\quad & $\leq 1$ \quad\quad & $\geq 4$ \quad\quad & & $\times$\\
				\quad\quad & $>1$ \quad\quad & $\geq 4$ \quad\quad & & \checkmark\\
				\bottomrule[0.4mm]
				Type I, Ramified \quad\quad & $\leq 1$ \quad\quad & $\geq 																																																																																																																																																																																																																																																																																																																																																																																																																																																																																																																																																																																																																																																																							3$ \quad\quad & &  $\times$\\
				\quad\quad & $2$ \quad\quad & $3$ \quad\quad & & $\times$\\
				\quad\quad & $>1$ \quad\quad & $\geq 5$ \quad\quad & & \checkmark\\
				\bottomrule[0.4mm]
				Type II  \quad\quad & & $\geq 2$ \quad\quad & & \checkmark\\
				(Ramified \& Unramified) & & & & \\ 
				\bottomrule[0.4mm]
				Type III, Unramified & $\leq1$ \quad\quad & $2$ \quad\quad & $p\geq5$& \checkmark \\
				& $\leq1$ \quad\quad& $\geq 4$ \quad\quad& $p\geq5$ & $\times$\\
				& $>1$ \quad\quad & $\geq 4$ \quad\quad & $p\geq5$ & \checkmark\\
				& $\leq1$ \quad\quad & $2$ \quad\quad & $p=3, \varkappa^2|_{\mco_K^\times} = 1$, & \checkmark\\
				& \quad\quad & \quad\quad & $p=3, \varkappa^2|_{\mco_K^\times} \neq 1$, & $\times$\\
				& \quad\quad & $\geq4$ \quad\quad & $p=3$ & \checkmark \\
				\bottomrule[0.4mm]
				Type III, Ramified & \quad\quad & $\geq 3$ \quad\quad & $p\geq5$& $\times$ \\
				\quad\quad&  & $\geq 3$ \quad\quad & $p=3$ &  \checkmark \\
				\bottomrule[0.4mm]

			\end{tabular}
			\vskip 1mm
			\caption{}
			\label{tab:table6}
		\end{table}
	}


\end{proof}

	\subsubsection{Computation of $\varepsilon_p$} \label{computation} Note that $(\Ind_{W(K)}^{W(\Q_p)}\varkappa )\chi_p= \Ind_{W(K)}^{W(\Q_p)} (\varkappa\chi_p')$ with $\chi_p'=\chi_p \circ N_{K|\Q_p}$. Using the Property $(\epsilon 3)$ of epsilon factors, when $\mathrm{sym}^3(\pi_p)$ is of either Type I or Type II, we have 
	\begin{equation}  \label{unr}
	\varepsilon_p = \frac{\varepsilon(\varkappa^3 \chi_p', \phi_K, dx) ~~\varepsilon(\varkappa \epsilon_p' \chi_p', \phi_K, dx)}{\varepsilon(\varkappa^3, \phi_K, dx) ~~ \varepsilon(\varkappa \epsilon_p', \phi_K, dx)}.
	\end{equation}
	For a Type III representation, we have 
	\begin{equation} \label{unr1}
	\varepsilon_p = \frac{\varepsilon(\varphi \omega_{K/\Q_p} \chi_p, \phi, dx) ~~\varepsilon(\varphi \chi_p, \phi, dx)~~\varepsilon(\varkappa \epsilon_p' \chi_p', \phi_K, dx)}{\varepsilon(\varphi \omega_{K/\Q_p}, \phi, dx) ~~\varepsilon(\varphi, \phi, dx) ~~ \varepsilon(\varkappa \epsilon_p', \phi_K, dx)}.
	\end{equation}
	For $K/\Q_p$ quadratic, when $a(\varkappa^3), a(\varkappa \epsilon_p')>1$, following Lemma \ref{relation between multiplicative and additive character} we consider an additive character $\phi$ of $\Q_p$ of conductor $-1$ satisfying 
	\begin{equation} \label{character1} \varkappa^3(1+x)=\phi_K(s_1x) \text{ for } x\in\mathfrak{p}_K^{r_1} \text{ with } 2r_1 \geq a(\varkappa^3),
	\end{equation} 
	\begin{equation} \label{character2}
	(\varkappa \epsilon_p')(1+x)=\phi_K(s_2x) \text{ for } x\in\mathfrak{p}_K^{r_2} \text{ with } 2r_2 \geq a(\varkappa \epsilon_p'),
	\end{equation}
	for some $s_1$ and $s_2$ with valuations $-a(\varkappa^3)+1$ and $-a(\varkappa \epsilon_p')+1$ respectively. Set $s: =s_1 s_2$. When $a(\varphi)>1$, we consider an additive character $\phi$ of $\Q_p$ of conductor $-1$ satisfying 
	\begin{equation} \label{character}
	(\varphi\omega^i_{K/\Q_p})(1+x)=\phi(d_ix) \text{ for } x\in\mathfrak{p}_{\Q_p}^r, ~~ i \in \{1, 2\}
	\end{equation}
	with $2r \geq a(\alpha)$ and $d_i$ has valuation $-a(\varphi)+1$. Set $d: =d_1d_2$. The Theorem below determines the variance number $\varepsilon_p$ depending upon $K/\Q_p$ is unramified or ramified.
	\begin{theorem} \label{supthm}
		Let $p$ be an odd dihedral supercuspidal prime for a newform $f$. We have $\pi_p:=\pi_{f, p}=\Ind_{W(K)}^{W(\Q_p)} (\varkappa)$, where $K/\Q_p$ is a quadratic extension. 
		Then, the variance number $\varepsilon_p$ is given as follows:
		\begin{enumerate}
			\item 
			Let $K/\Q_p$ be unramified with $a(\varkappa)>1$ and $C_p \neq \frac{N_p}{2}$. Assume $a(\varkappa^3)>1$ if $p=3$. Then 
			\begin{equation} \label{punramified}
			\varepsilon_p =
			\begin{cases}
			\chi_p'(s), 
			\quad \quad \quad \,\,\,\,\, \text{if} \,\, \mathrm{sym}^3(\pi_p) \text{ is either Type I or II},  \\
			\chi_p(d) \chi_p'(s_2),
			\quad \text{if} \,\, \mathrm{sym}^3(\pi_p) \text{ is Type III}. 
			\end{cases}
			\end{equation}
			If $a(\varkappa^3) \leq 1$, then $p=3$. For $a(\varkappa^3)=0$, we have $\varepsilon_3 =\varkappa^3(3) \chi_3'(s_2)$.
			Let $a(\varkappa^3)=1$. For Type I or II representations, $\varepsilon_3 =\frac{ \chi_3'(s_2)}{\varkappa^3(3)}$ or $\chi_3'(s_2)$ if $\circ(\varkappa^3|_{\mco_K^\times})=2$ or $4$ respectively. For a Type III representation, $\varepsilon_3=\frac{\chi_3'(s_2)}{\varkappa^3(3)}$.
			
			Here, $s$, $s_2$ and $d$ have valuations $-a(\varkappa)+1$, $-a(\varkappa)-a(\varkappa^3)+2$ and $-2a(\varkappa^3)+2$ respectively.
			\item 
			Assume that $K/\Q_p$ is ramified. For Type I and II representations, we have $\varepsilon_p=1$ when $p \geq 5$ or $p=3$ with $a(\varkappa) \equiv a(\varkappa^3) \pmod{2}$. If $p=3$ with $a(\varkappa) \not\equiv a(\varkappa^3) \pmod{2}$, then we have 
			\begin{equation} \label{pramified}
			\varepsilon_3 =
			\begin{cases}
			1, 
			\quad \text{if} \,\, (3, K/\Q_3)=1,  \\
			-1,
			\quad \text{if} \,\, (3, K/\Q_3)=-1.
			\end{cases}
			\end{equation}
			Let $\sym^3(\pi_p)$ be of Type III. Then $p=3$. If $a(\varkappa^3)>1$ then we have
			\begin{equation} 
			\varepsilon_3 =
			\begin{cases}
			\chi_3(d), 
			\quad \text{if} \,\, (3, K/\Q_3)=1,  \\
			-\chi_3(d) ,
			\quad \text{if} \,\, (3, K/\Q_3)=-1,
			\end{cases}
			\end{equation}
			where $d$ has valuation $-a(\varkappa^3)+1$. Otherwise, $\varepsilon_3=1$. 
		\end{enumerate}
	\end{theorem}

	\subsubsection{$K/\Q_p$ unramified}	Note that $a(\chi_p')=1$ and $n(\phi_K)=-1$  by Lemma \ref{tunnel}. We now give a proof of Theorem \ref{supthm} in the unramified case.
	
	\smallskip 
	
	\noindent 
	{\bf Proof of Theorem \ref{supthm}(1):} Let $\mathrm{sym}^3(\pi_p)$ be of either Type I or Type II. Equ. \ref{unr} determines $\varepsilon_p$. By assumption, $a(\varkappa), a(\varkappa^3)>1$. As $C_p \neq \frac{N_p}{2}$ by assumption, from Prop. \ref{prop7} recall that $a(\varkappa \epsilon_p')>1$ and so we can use the additive characters as in \eqref{character1} and \eqref{character2}. Now by Theorem \ref{epsilon factor while twisting} together with the fact $\chi_p'$ is quadratic with conductor $1$, we have
	\begin{enumerate}
		\item 
		$\varepsilon(\varkappa^3 \chi_p', \phi_K, dx)  = \chi'_p(s_1) \varepsilon(\varkappa^3, \phi_K, dx) $
		\item 
		$\varepsilon(\varkappa \epsilon_p' \chi_p', \phi_K, dx)  = \chi_p'(s_2) \varepsilon(\varkappa \epsilon_p', \phi_K, dx)$
	\end{enumerate}
	Therefore, using \ref{unr} we obtain that $\varepsilon_p=\chi_p'(s)$.

	On the other hand, for a Type III representation, Equ. \ref{unr1} determines $\varepsilon_p$. As $a(\varphi)=a(\varkappa^3)>1$ by Lemma \ref{phi} and $\omega_{K/\Q_p}$ is unramified, we have $a(\varphi \omega_{K/\Q_p}) =a(\varphi)$. Also, $\chi_p$ is tamely ramified implies that $a(\varphi\chi_p) =a(\varphi)$.
	As $a(\varphi)>1$, we use the additive chracter $\phi$ defined in \eqref{character}. Thus, by Theorem \ref{epsilon factor while twisting} we obtain that
	\begin{enumerate}
		\item 
		$\varepsilon(\varphi \omega_{K/\Q_p} \chi_p, \phi, dx)  = \chi_p(d_1) \varepsilon(\varphi \omega_{K/\Q_p}, \phi, dx) $
		\item 
		$\varepsilon(\varphi \chi_p, \phi, dx)  = \chi_p(d_2) \varepsilon(\varphi, \phi, dx)$
	\end{enumerate}
	Therefore, it follows that $\varepsilon_p = \chi_p(d) \chi_p'(s_2)$.
	
	\qed

	Let us denote 
	\begin{equation} \label{aphi}
	A_\varphi:=A(\varphi, \omega_{K/\Q_p}, \chi_p)= \frac{\varepsilon(\varphi \omega_{K/\Q_p} \chi_p, \phi, dx) ~~\varepsilon(\varphi \chi_p, \phi, dx)}{\varepsilon(\varphi \omega_{K/\Q_p}, \phi, dx) ~~ \varepsilon(\varphi, \phi, dx)}.
	\end{equation}
	
	\begin{lemma} \label{aphi1}
		Let $K/\Q_p$ be unramified quadratic. If $\varphi$ is unramified, then the quantity
		\begin{eqnarray*} 
			A_\varphi=
			\begin{cases}
				\varphi^2(p)  \omega_{K/\Q_p}(p), & \quad \text{if} \,\, p \equiv 1 \pmod{4}, \\
				-\varphi^2(p)  \omega_{K/\Q_p}(p), & \quad \text{if} \,\, p \equiv 3 \pmod{4}.
			\end{cases}
		\end{eqnarray*}
		Suppose $\varphi$ is tamely ramified. If $\varphi^2|_{\Z_p^\times} \neq 1$, then $A_\varphi=-p \Big \{\Gamma_p\Big(\frac{1}{n}+\frac{1}{2}\Big)/\Gamma_p\Big(\frac{1}{n}\Big) \Big \}^2$ where $n=\circ(\widetilde{\varphi})$; otherwise 
		\begin{eqnarray*} 
			A_\varphi=
			\begin{cases}
				\frac{1}{\varphi^2(p)\omega_{K/\Q_p}(p)}, & \quad \text{if} \,\, p \equiv 1 \pmod{4}, \\
				\frac{-1}{\varphi^2(p)\omega_{K/\Q_p}(p)}, & \quad \text{if} \,\, p \equiv 3 \pmod{4}.
			\end{cases}
		\end{eqnarray*}
	\end{lemma}
	
	\begin{proof}
		First, assume that $\varphi$ is unramified. As $\omega_{K/\Q_p}$ is also unramified, using $(\epsilon 2)$, $\varepsilon(\varphi\omega_{K/\Q_p}\chi_p, \phi)=\varepsilon(\chi_p, \phi)$. Also, $\varepsilon(\varphi \omega_{K/\Q_p}, \phi)=\frac{1}{(\varphi \omega_{K/\Q_p})(p)}$. Similarly, $\varepsilon(\varphi\chi_p, \phi)=\varepsilon(\chi_p, \phi)$ and $\varepsilon(\varphi, \phi)=\frac{1}{\varphi(p)}$. Thus, by Lemma \ref{alpha} we deduce the desired value of $A_\varphi$.
		
		Now, assume that $a(\varphi)=1$. As $\omega_{K/\Q_p}$ is unramified, by Property $(\epsilon 2)$ of epsilon factors we have that 
		\begin{equation} \label{un1}
		\varepsilon(\varphi \omega_{K/\Q_p} \chi_p, \phi, dx)=\omega_{K/\Q_p}(p)^{a(\varphi \chi_p)-1} \varepsilon(\varphi\chi_p, \phi, dx)
		\end{equation}
		\begin{equation} \label{un2}
		\varepsilon(\varphi \omega_{K/\Q_p}, \phi, dx)=\omega_{K/\Q_p}(p)^{a(\varphi)-1} \varepsilon(\varphi, \phi, dx)
		\end{equation}
		\begin{enumerate} 
			\item
			First, consider that $\varphi^2|_{\Z_p^\times}=1$. As both $\varphi$ and $\chi_p$ are quadratic characters of conductor $1$, by Lemma \ref{alphabeta} we have $\varepsilon(\varphi \chi_p, \phi)=1/\varphi(p)$. Now,
			\begin{eqnarray*}
				A_\varphi &=& \frac{1}{\omega_{K/\Q_p}(p)} \Big\{ \frac{\varepsilon(\varphi \chi_p, \phi, dx)}{\varepsilon(\varphi, \phi, dx)} \Big\}^2  \\
				&=& \frac{1}{\omega_{K/\Q_p}(p)} \Big\{ \frac{1}{\varphi(p) ~\varepsilon(\varphi, \phi)} \Big\}^2  \\
				&\overset{\text{Lemma} ~\ref{alpha}}=&
				\begin{cases}
					\frac{1}{\varphi^2(p)\omega_{K/\Q_p}(p)}, & \quad \text{if} \,\, p \equiv 1 \pmod{4}, \\
					\frac{-1}{\varphi^2(p)\omega_{K/\Q_p}(p)}, & \quad \text{if} \,\, p \equiv 3 \pmod{4}.
				\end{cases}
			\end{eqnarray*}
			\item 
			Next, we assume that $\varphi^2|_{\Z_p^\times} \neq 1$. In this case, $a(\varphi \chi_p)=1$. Let $\widetilde{\varphi}$ has order $n$ and write $\widetilde{\varphi}= \chi_1^a$, where $\widehat{\mathbb{F}_p^\times}=\langle \chi_1 \rangle$. Then $\widetilde{\varphi}\widetilde{\chi_p}=\chi_1^{a+\frac{p-1}{2}}$. Thus, by definition of epsilon factors and Equ. \eqref{GKcoro}, $\varepsilon(\varphi, \phi)=p^{-1/2}G_1(\chi_1^a)=p^{-1/2}(-p)^{\frac{a}{p-1}} \Gamma_p(\frac{a}{p-1})=p^{-1/2}(-p)^{\frac{1}{n}} \Gamma_p(\frac{1}{n})$. Similarly, $\varepsilon(\varphi \chi_p, \phi)=p^{-1/2}(-p)^{\frac{1}{n}+\frac{1}{2}} \Gamma_p(\frac{1}{n}+\frac{1}{2})$.
			Using \eqref{un1} and \eqref{un2}, we have
			\begin{eqnarray*}
				A_\varphi &=& \Big\{ \frac{\varepsilon(\varphi \chi_p, \phi, dx)}{\varepsilon(\varphi, \phi, dx)} \Big\}^2\\
				&=& -p \Big \{\Gamma_p\Big(\frac{1}{n}+\frac{1}{2}\Big)/\Gamma_p\Big(\frac{1}{n}\Big) \Big \}^2.
			\end{eqnarray*}
		\end{enumerate}
	\end{proof}
	\noindent
	{\bf Proof of the case} {\boldmath $a(\varkappa)>1, a(\varkappa^3)\leq 1$.} This is possible only when $p=3$ (cf. Prop. \ref{chi33}). First, we assume that $a(\varkappa^3)=0$. By Property $(\epsilon 2)$, $\varepsilon(\varkappa^3 \chi_p', \phi_K)=\varepsilon(\chi_p', \phi_K)$. As $p=3$ and $a(\chi_p')=1$, by the argument used in the proof of \cite[Lemma 2.3]{bm} we obtain that $\varepsilon(\chi_p', \phi_K)=p^{-1}G(\widetilde{\chi_p'}, \widetilde{\phi_K})=\frac{p}{p}=1$ by \cite[Theorem 5.15]{MR1294139}. Also, $\varepsilon(\varkappa^3, \phi_K)=\frac{1}{\varkappa^3(p)}$. Then for Type I or II representations, we have from \eqref{unr} that $\varepsilon_3=\varkappa^3(3) \chi_3'(s_2)$. On the other hand, when $\sym^3(\pi_p)$ is of Type III, using \eqref{unr1} we have $\varepsilon_p=A_\varphi \times \frac{\varepsilon(\varkappa \epsilon_p' \chi_p', \phi_K)}{\varepsilon(\varkappa \epsilon_p', \phi_K)}$. Thus, by Lemma \ref{aphi1}, $\varepsilon_3=-\varphi^2(3)\omega_{K/\Q_3}(3)\chi_3'(s_2)=\varkappa^3(3)\chi_3'(s_2)$ as $\omega_{K/\Q_3}(3)=-1$.
	
	Next, assume that $a(\varkappa^3)=1$. If $m:=\circ(\varkappa^3|_{\mco_K^\times})=2$, then $\varkappa^3\chi_p'$ is unramified. Hence, $\varepsilon(\varkappa^3 \chi_p', \phi_K)=\frac{1}{\varkappa^3(p)}$ as $\chi_p'(p)=1$. As $\varkappa^3|_{\mco_K^\times}$ is quadratic, proceeding as before we obtain that $\varepsilon(\varkappa^3, \phi_K)=1$. Thus, for Type I or II representations, we have from \eqref{unr} that $\varepsilon_3=\frac{ \chi_3'(s_2)}{\varkappa^3(3)}$. On the other hand, if $m>2$, then $\varkappa^3\chi_p'$ is tamely ramified. In this case, $\frac{\varepsilon(\varkappa^3 \chi_p', \phi_K)}{\varepsilon(\varkappa^3, \phi_K)}=\frac{G(\widetilde{\varkappa^3} \widetilde{\chi_p'}, \widetilde{\phi_K})}{G(\widetilde{\varkappa^3}, \widetilde{\phi_K})}$. Note that $m=4, 8$. If $m=4$, then $m \mid (p+1)$ and $\frac{p+1}{m}=1$ which is odd. Thus, using Stickelberger’s theorem \cite[Theorem 5.16]{MR1294139} we deduce that $\frac{G(\widetilde{\varkappa^3} \widetilde{\chi_p'}, \widetilde{\phi_K})}{G(\widetilde{\varkappa^3}, \widetilde{\phi_K})}=1$. Therefore, from Equ. \ref{unr}, $\varepsilon_3=\chi_3'(s_2)$. When $m=8$ we cannot compute $\varepsilon_3$ using Stickelberger’s theorem, because $m \nmid (p+1)$. For a Type III representation, $a(\varphi)=a(\varkappa^3)=1$ by Lemma \ref{phi} and so we have $\varphi^2|_{\Z_p^\times} =1$ as $p=3$. Thus, using Lemma \ref{aphi1} we obtain that $\varepsilon_3= \frac{\chi_3'(s_2)}{\varkappa^3(3)}$. 
	
	\bigskip 
	

	\noindent
	{\bf The case \boldmath $a(\varkappa)=1$.} 
	Let $\circ(\varkappa|_{\mco_K^\times}) = m$ and $\circ(\epsilon_p'|_{\mco_K^\times}) =m'$. We claim that $m' \mid m$. Assume that $\circ(\epsilon_p |_{\Z_p^\times}) =t$. Since $\epsilon_p'=\epsilon_p \circ N_{K/\Q_p}$, we obtain that $m' \mid t$.  Using Remark \ref{centralcharacter}, we have $\varkappa|_{\mathbb{Z}_p^\times}= \epsilon_p^{-1}|_{\mathbb{Z}_p^\times}$. Hence $\circ(\varkappa|_{\mathbb{Z}_p^\times})=t$. As $(\varkappa|_{\mco_K^\times})^m = 1$ implies that $(\varkappa|_{\Z_p^\times})^m = 1$, we deduce $t \mid m$ and hence $m' \mid m$, proving our cliam.
	
	\begin{theorem} \label{supthm1}
		Let $p$ be an odd dihedral supercuspidal prime for $f$ with $N_p=2$. The number $\varepsilon_p$ is given as follows:
		\begin{enumerate}
			\item 
			Let $\sym^3(\pi_p)$ be of Type I or II. If $m \mid (p-1)$, then 
			\begin{equation} \label{achi1}
			\varepsilon_p = p^2 \Big \{ \Big[ \Gamma_p\Big(\frac{1}{m}+\frac{1}{m'}+\frac{1}{2}\Big) \Gamma_p\Big(\frac{3}{m}+\frac{1}{2}\Big) \Big]/ \Big[\Gamma_p\Big(\frac{1}{m}+\frac{1}{m'}\Big) \Gamma_p\Big(\frac{3}{m}\Big) \Big] \Big \}^2.
			\end{equation}
			Suppose that $m \mid (p+1)$. When $m$ is even and $p \equiv 3 \pmod{4}$, we further assume $\frac{p+1}{m}$ is odd. Then, $\varepsilon_p=1$.
			\item 
			Let $\sym^3(\pi_p)$ be of Type III. We have:
			
			(i) Assume $\varkappa^3$ is unramified. If $m \mid (p-1)$, then 
			\begin{eqnarray} 
			\varepsilon_p =
			\begin{cases}
			b_p\varphi^2(p)  \omega_{K/\Q_p}(p), & \quad \text{if} \,\, p \equiv 1 \pmod{4}, \\
			-b_p\varphi^2(p)  \omega_{K/\Q_p}(p), & \quad \text{if} \,\, p \equiv 3 \pmod{4},
			\end{cases}
			\end{eqnarray}
			where $b_p:=-p \Big[ \Gamma_p(\frac{1}{m}+\frac{1}{m'}+\frac{1}{2}) /   \Gamma_p(\frac{1}{m}+\frac{1}{m'})  \Big]^2$. If $m \mid (p+1)$, then 
			\begin{eqnarray} 
			\varepsilon_p=
			\begin{cases}
			-\varphi^2(p)  \omega_{K/\Q_p}(p), & \quad \text{if} \,\, m \,\,
			\text{odd and} \,\, p \equiv 1 \pmod{4}, \\
			\varphi^2(p)  \omega_{K/\Q_p}(p), & \quad \text{if} \,\, m \,\,
			\text{even and} \,\, p \equiv 1 \pmod{4},   \\
			-\varphi^2(p)  \omega_{K/\Q_p}(p), & \quad \text{if} \,\, m \,\,
			\text{even and} \,\, p \equiv 3 \pmod{4} \,\, \text{with} \,\, 
			\frac{p+1}{m} \,\, \text{odd}.
			\end{cases}
			\end{eqnarray}
			
			(ii) Suppose $\varkappa^3$ is tamely ramified. If $m \mid (p-1)$, then we have
			\begin{eqnarray} 
			\varepsilon_p=
			\begin{cases}
			\frac{b_p}{\varphi^2(p) \omega_{K/\Q_p}(p)}, & \quad \text{if} \,\, \varphi^2|_{\Z_p^\times} =1, \\
			-pb_p \Big \{\Gamma_p\Big(\frac{1}{n}+\frac{1}{2}\Big)/\Gamma_p\Big(\frac{1}{n}\Big) \Big \}^2, & \quad \text{if} \,\, \varphi^2|_{\Z_p^\times} \neq 1.
			\end{cases}
			\end{eqnarray}
			
			Let $m \mid (p+1)$. If $m$ is odd and $p \equiv 1 \pmod{4}$, then 
			\begin{eqnarray} 
			\varepsilon_p=
			\begin{cases}
			-\frac{1}{\varphi^2(p)\omega_{K/\Q_p}(p)}, & \quad \text{if} \,\, \varphi^2|_{\Z_p^\times} =1, \\
			-p \Big \{\Gamma_p\Big(\frac{1}{n}+\frac{1}{2}\Big)/\Gamma_p\Big(\frac{1}{n}\Big) \Big \}^2, & \quad \text{if} \,\, \varphi^2|_{\Z_p^\times} \neq 1.
			\end{cases}
			\end{eqnarray}
			If $m$ is even with $\frac{p+1}{m}$ odd, then
			\begin{eqnarray} 
			\varepsilon_p=
			\begin{cases}
			\frac{1}{\varphi^2(p)\omega_{K/\Q_p}(p)}, & \quad \text{if} \,\, p \equiv 1 \pmod{4},~ \varphi^2|_{\Z_p^\times} =1, \\
			-\frac{1}{\varphi^2(p)\omega_{K/\Q_p}(p)}, & \quad \text{if} \,\, p \equiv 3 \pmod{4}, ~ \varphi^2|_{\Z_p^\times} =1, \\
			p \Big \{\Gamma_p\Big(\frac{1}{n}+\frac{1}{2}\Big)/\Gamma_p\Big(\frac{1}{n}\Big) \Big \}^2, & \quad \text{if} \,\, \varphi^2|_{\Z_p^\times} \neq 1.
			\end{cases}
			\end{eqnarray}
		\end{enumerate}
	\end{theorem}

	\begin{proof}
		(1)For Type I and II representations, Equ.\ref{unr} determines $\varepsilon_p$. 
		Set $m:=\circ(\widetilde{\varkappa})$ and $m':=\circ(\widetilde{\epsilon_p'})$.
		We split the proof into two cases.
		
		(i) Suppose that $m \mid (p-1)$. Then $\widetilde{\varkappa}$ can be thought of as a lift of some character $\widetilde{\varkappa}^* \in \widehat{\mathbb{F}_p^\times}$ such that $\circ(\widetilde{\varkappa})=\circ(\widetilde{\varkappa}^*)$. As $m' \mid m$, we have the same thing for $\widetilde{\epsilon_p'}$ also. Now, applying Davenport-Hasse theorem from \S \ref{Gauss}, we get that
		\begin{equation} \label{unr2}
		\frac{G_2(\widetilde{\varkappa} \,\widetilde{\epsilon_p'} \, \widetilde{\chi_p'})}{G_2(\widetilde{\varkappa} \,\widetilde{\epsilon_p'})} = \left\{ \frac{G_1\left(\widetilde{\varkappa}^*  (\widetilde{\epsilon_p'})^*  (\widetilde{\chi_p'})^* \right)}{G_1\left(\widetilde{\varkappa}^* (\widetilde{\epsilon_p'})^* \right)}  \right \}^2
		\end{equation}
		Let $\widetilde{\varkappa}^*= \chi_1^b$ where $\widehat{\mathbb{F}_p^\times}=\langle \chi_1 \rangle$. Also, $\widetilde{\epsilon_p'}^*= \chi_1^{\frac{bm}{m'}}$. Then $\widetilde{\varkappa}^* (\widetilde{\epsilon_p'})^* (\widetilde{\chi_p'})^*=\chi_1^{b+\frac{bm}{m'}+\frac{p-1}{2}}$ and $\widetilde{\varkappa}^* (\widetilde{\epsilon_p'})^*=\chi_1^{b+\frac{bm}{m'}}$. By Equations \ref{GKcoro} and \ref{unr2}, $\frac{\varepsilon(\varkappa \epsilon_p' \chi_p', \phi_K)}{\varepsilon(\varkappa \epsilon_p', \phi_K)}=\Big[(-p)^{\frac{1}{m}+\frac{1}{m'}+\frac{1}{2}} \Gamma_p(\frac{1}{m}+\frac{1}{m'}+\frac{1}{2}) / (-p)^{\frac{1}{m}+\frac{1}{m'}}  \Gamma_p(\frac{1}{m}+\frac{1}{m'})  \Big]^2$. Thus,
		\begin{equation} \label{q2}
		\frac{\varepsilon(\varkappa \epsilon_p' \chi_p', \phi_K)}{\varepsilon(\varkappa \epsilon_p', \phi_K)} = -p \Big[ \Gamma_p(\frac{1}{m}+\frac{1}{m'}+\frac{1}{2}) /   \Gamma_p(\frac{1}{m}+\frac{1}{m'})  \Big]^2.
		\end{equation}
		Also, $\widetilde{\varkappa^3}^*= \chi_1^{3b}$. By similar calculations, $\frac{G_2(\widetilde{\varkappa^3} \widetilde{\chi_p'})}{G_2(\widetilde{\varkappa^3})} = -p \Big[ \Gamma_p(\frac{3}{m}+\frac{1}{2}) /   \Gamma_p(\frac{3}{m})  \Big]^2$. Thus, from Equ. \ref{unr}, 
		\begin{equation*} 
		\varepsilon_p = p^2 \Big \{ \Big[ \Gamma_p\Big(\frac{1}{m}+\frac{1}{m'}+\frac{1}{2}\Big) \Gamma_p\Big(\frac{3}{m}+\frac{1}{2}\Big) \Big]/ \Big[\Gamma_p\Big(\frac{1}{m}+\frac{1}{m'}\Big) \Gamma_p\Big(\frac{3}{m}\Big) \Big] \Big \}^2.
		\end{equation*}
		
		(ii) Next, suppose that $m \mid (p+1)$. As $m' \mid m$, we have $\circ(\widetilde{\varkappa}\widetilde{\epsilon_p'})=m$. If $m$ is odd and $p\equiv 1 \pmod 4$, then using Stickelberger’s theorem \cite[Theorem 5.16]{MR1294139} we have $G(\widetilde{\varkappa} \widetilde{\epsilon_p'}, \widetilde{\phi})=p$. Note that $\circ(\widetilde{\varkappa}\widetilde{\epsilon_p'}\widetilde{\chi_p})=2m$. As $\frac{p+1}{2m}$ is odd, by the same theorem we get $G(\widetilde{\varkappa} \widetilde{\epsilon_p'} \widetilde{\chi_p}, \widetilde{\phi})=-p$. Thus, $\frac{G(\widetilde{\varkappa} \widetilde{\epsilon_p'} \widetilde{\chi_p}, \widetilde{\phi})} {G(\widetilde{\varkappa} \widetilde{\epsilon_p'}, \widetilde{\phi})}=-1$. Similarly, $\frac{G(\widetilde{\varkappa^3}\widetilde{\chi_p}, \widetilde{\phi})} {G(\widetilde{\varkappa^3}, \widetilde{\phi})}=-1$. Hence, $\varepsilon_p=1$. If $p\equiv 3 \pmod 4$ with $m$ odd, then as $\frac{p+1}{2m}$ is not odd, we cannot apply Stickelberger’s theorem to find $G(\widetilde{\varkappa}\widetilde{\epsilon_p'}\widetilde{\chi_p}, \widetilde{\phi})$.
		
		Now, assume that $m$ is even ($\neq 2$) $\Rightarrow \circ(\widetilde{\varkappa}\widetilde{\epsilon_p'}\widetilde{\chi_p})=m$. For $p \equiv 1 \pmod 4$, we have that $\frac{p+1}{m}$ is odd. Thus, by Stickelberger’s theorem, $G(\widetilde{\varkappa} \widetilde{\epsilon_p'}, \widetilde{\phi})=G(\widetilde{\varkappa} \widetilde{\epsilon_p'} \widetilde{\chi_p}, \widetilde{\phi})=-p$ and $G(\widetilde{\varkappa^3}, \widetilde{\phi})=G(\widetilde{\varkappa^3}\widetilde{\chi_p}, \widetilde{\phi})=-p$.  Hence, $\varepsilon_p=1$. The same things hold for primes $p \equiv 3 \pmod 4$ with $\frac{p+1}{m}$ odd. The case $m=2$ is already covered in part(i).

		\smallskip 
		
		\noindent
		(2) For Type III representations, $\varepsilon_p=A_\varphi \times \frac{\varepsilon(\varkappa \epsilon_p' \chi_p', \phi_K)}{\varepsilon(\varkappa \epsilon_p', \phi_K)}$. The proof will be divided into two cases, $\varphi$ is unramified or $a(\varphi)=1$. Then we use Lemma \ref{aphi1} and Davenport-Hasse theorem or Stickelberger’s theorem depending upon $m \mid (p-1)$ or $m \mid (p+1)$ respectively to obtain the result.
	\end{proof}

	\subsubsection{$K/\Q_p$ ramified}  For an odd prime $p$, the ramified quadratic extensions of $\Q_p$ is characterized as follows: $K=\Q_p(\sqrt{-p})$ or $\Q_p(\sqrt{-p\zeta_{p-1}})$ depending upon $(p, K/\Q_p)=1$ or $-1$ respectively. Here, $\zeta_{p-1}$ denotes a primitive $(p-1)$-th root of unity. Choosing $\pi=\sqrt{-p}$ or $\sqrt{-p\zeta_{p-1}}$ as a uniformizer of $K$, we write $K=\Q_p(\pi)$. Recall, $\chi_p'=\chi_p \circ N_{K/\Q_p}$. As $N_{K/\Q_p}(\mathcal{O}_K^\times)=(\Z_p^\times)^2$, the character $\chi_p'$ is unramified. By Lemma \ref{tunnel}, we have $n(\phi_K)=-1$ as $n(\phi)=-1$.
	
	\smallskip 
	
	\noindent 
	{\bf Proof of Theorem \ref{supthm}(2):} We now give a proof of the second part of Theorem \ref{supthm}.
	\begin{proof}
		As the character $\chi_p'$ is unramified, using Property $(\epsilon 2)$ of epsilon factors we have
		\begin{equation} \label{rm1}
		\varepsilon(\varkappa^3 \chi_p', \phi_K, dx) =\chi_p(N_{K/\Q_p}(\pi))^{a(\varkappa^3)-1} \cdot \varepsilon(\varkappa^3, \phi_K, dx)
		\end{equation}
		\begin{equation} \label{rm2}
		\varepsilon(\varkappa \epsilon_p' \chi_p', \phi_K, dx)=\chi_p(N_{K/\Q_p}(\pi))^{a(\varkappa \epsilon_p')-1} \cdot \varepsilon(\varkappa \epsilon_p', \phi_K, dx).
		\end{equation}
		For Type I and II representations, by Property $(\epsilon 3)$ of epsilon factors and noting $\chi_p$ is a quadratic character we have that 
		\begin{eqnarray} \label{e1}
		\varepsilon_p = \chi_p(N_{K/\Q_p}(\pi))^{a(\varkappa^3)+a(\varkappa \epsilon_p')}. 
		\end{eqnarray}
		By the proof of Prop. \ref{prop7}, $a(\varkappa\epsilon_p')=a(\varkappa)$. If $p\geq 5$, then from Prop. \ref{chi3} we have that $a(\varkappa^3)=a(\varkappa)$ as $a(\varkappa)>1$. Thus, we deduce from \eqref{e1} that $\varepsilon_p=1$. Now, we consider $p=3$. If the conductors of both $\varkappa$ and $\varkappa^3$ have the same parity, then $\varepsilon_3=1$. If they have the different parity, then we get that
		\begin{eqnarray} \label{variation}
		\varepsilon_p=\chi_p(N_{K/\Q_p}(\pi))=\Big( \frac{N_{K/\Q_p}(\pi)/p}{p} \Big).
		\end{eqnarray} 
		Since $N_{K/\Q_p}(\pi)=-\pi^2$, we obtain that
		\begin{equation} \label{v1}
		\varepsilon_p=\Big( \frac{-\pi^2/p}{p} \Big).
		\end{equation}
		Thus, for the prime $p=3$ the variation number is given as follows.
		\begin{equation} \label{residuesymbolp=3}
		\varepsilon_3 =
		\begin{cases}
		1, 
		\quad \text{if} \,\, (3, K/\Q_3)=1,  \\
		-1,
		\quad \text{if} \,\, (3, K/\Q_3)=-1.
		\end{cases}
		\end{equation}
		
		\medskip 
		
		Now, we consider the Type III representations. By Lemma \ref{phi}, we have $p=3$. Using Equations \ref{unr1} and \ref{rm2}, we deduce that
		\begin{eqnarray}  \label{epp2}
		\varepsilon_p =  \frac{\varepsilon(\varphi \omega_{K/\Q_p} \chi_p, \phi, dx) ~~\varepsilon(\varphi \chi_p, \phi, dx)}{\varepsilon(\varphi \omega_{K/\Q_p}, \phi, dx) ~~\varepsilon(\varphi, \phi, dx)} \chi_p(N_{K/\Q_p}(\pi))^{a(\varkappa \epsilon_p')-1}.
		\end{eqnarray}

		Assume that $a(\varkappa^3)>1$. As $\varkappa^3=\varphi \circ N_{K/\Q_p}$, by Lemma \ref{tunnel} $a(\varkappa^3)>1 \iff a(\varphi)>1$. As $K/\Q_p$ is ramified quadratic, $a(\omega_{K/\Q_p})=1$. Also, $\chi_p$ has conductor $1$. Therefore, $a(\varphi\omega_{K/\Q_p}) =a(\varphi)=a(\varphi\chi_p) $.
		We will use the additive chracter $\phi$ defined in \eqref{character}. Thus, by Theorem \ref{epsilon factor while twisting} we have that
		\begin{enumerate}
			\item 
			$\varepsilon(\varphi \omega_{K/\Q_p} \chi_p, \phi, dx)  = \chi_p(d_1) \varepsilon(\varphi \omega_{K/\Q_p}, \phi, dx) $ \smallskip 
			\item 
			$\varepsilon(\varphi \chi_p, \phi, dx)  = \chi_p(d_2) \varepsilon(\varphi, \phi, dx)$
		\end{enumerate}
		Combining these results with Equ. \ref{epp2} and noting $a(\varkappa\epsilon_p')=a(\varkappa)$, we obtain that 
		\begin{equation} 
		\varepsilon_3 = \chi_3(d) \chi_3(N_{K/\Q_3}(\pi))^{a(\varkappa)-1}.
		\end{equation} 
		For a $p$-minimal for $f$, $a(\varkappa)$ is even. Thus, we have $\varepsilon_3 = \chi_3(d) \Big( \frac{N(\pi)/3}{3} \Big)$.
		Thus, we obtain that
		\begin{equation} \label{residuesymbolp}
		\varepsilon_3 =
		\begin{cases}
		\chi_3(d), 
		\quad  \,\,\text{if} \,\, (3, K/\Q_3)=1,  \\
		-\chi_3(d),
		\quad \text{if} \,\, (3, K/\Q_3)=-1.
		\end{cases}
		\end{equation}

		\smallskip 
		
		Next, we assume that $a(\varkappa^3) \leq 1$. In this case, $a(\varphi)\leq 1$. Let $a(\varphi)=1$. As $p=3$, it follows that $\varphi^2|_{\Z_p^\times} =1$. Thus, both $\varphi \chi_p$ and $\varphi\omega_{K/\Q_p}$ are unramified and $a(\varphi \omega_{K/\Q_p}\chi_p)=1$. Hence, using Lemma \ref{alpha} and Equ. \ref{epp2},
		\begin{eqnarray*}
			\varepsilon_p &=& \frac{i \cdot  1/(\varphi\chi_p)(p)}{1/(\varphi\omega_{K/\Q_p})(p) \cdot i}  
			\chi_p(N_{K/\Q_p}(\pi))^{a(\varkappa \epsilon_p')-1} \\
			&=& \omega_{K/\Q_p}(p) \chi_p(N_{K/\Q_p}(\pi))^{a(\varkappa)-1} \quad [\text{since } \chi_p(p)=1 \text{ and } a(\varkappa \epsilon_p')=a(\varkappa)] 
		\end{eqnarray*}
		As $a(\varkappa)$ is even, we obtain that $\varepsilon_3=1$.

		\smallskip 
		
		Lastly, suppose $a(\varphi)=0$. Note $\varphi \omega_{K/\Q_p}\chi_p$ is unramified and both $\varphi \omega_{K/\Q_p}, \varphi \chi_p$ are quadractic tamely ramified. Thus, by Lemma \ref{alpha} and \eqref{epp2}, $\varepsilon_3=\omega_{K/\Q_3}(3) \chi_3(N_{K/\Q_3}(\pi))^{a(\varkappa)-1}=1$.
	\end{proof}

	By the variance number while twisting by $\chi_p$, \cite[Corollary 5.8]{bm} classifies the type of local representations $\pi_p$ at $p$ of a $p$-minimal newform. If $N_p$ is even, then consider the elements $s$, $s_2$ and $d$ of $\mco_K^\times$, as defined in \S \ref{computation}, with valuations $-a(\varkappa)+1$, $-a(\varkappa)-a(\varkappa^3)+2$ and $-2a(\varkappa^3)+2$ respectively. For $N_p$ odd, consider $d \in \mco_K^\times$ with valuation $-a(\varkappa^3)+1$, see Equ. \ref{character}.
	The corollary below does the same for symmetric cube transfers of $\pi_p$.
	\begin{cor} \label{mainkoro}
		Let $\pi_p=\pi_{f,p}$ be the local component at $p$ attached to a $p$-minimal newform $f$. 
		Then ${\rm sym}^3(\pi_p)$ is a representation of ${\rm GL}_4(\mathbb{Q}_p)$. Let $M'$ denote the prime-to-$p$ part of $a(\sym^3(\pi))$. We have the following:  
		\begin{enumerate}
			\item 
			$\sym^3(\pi_p)$ is a Special representation if $N_p=1$ and $C_p=0$.
			\item
			$\sym^3(\pi_p)$ is a principal series representation if $N_p \geq 1$ with $N_p=C_p$. 
			\item
			If $\sym^3(\pi_p)$ is not of the above type, then it is either of Type I, II or III. In this case, $2 \leq N_p > C_p$ and $\pi_p=\Ind_{W(K)}^{W(\Q_p)} \varkappa$ with $K/\Q_p$ quadratic. Then $\sym^3(\pi_p)=\Ind_{W(K)}^{W(\Q_p)} \varkappa ~~ \oplus ~~$ $\Ind_{W(K)}^{W(\Q_p)} (\varkappa \epsilon_p')$, where $\epsilon_p'=\epsilon_p^{-1}\circ N_{K/\Q_p}$. 
			
			\vskip 2mm
			
			\noindent If $N_p$ is even, then $K=\Q_p(\zeta_{p^2-1})$ is the unique unramified quadratic extension of $\Q_p$. 
			Then, for $C_p \neq \frac{N_p}{2}$, if $p \geq 5$ then we have the following:
			\begin{enumerate}
				\item 
				$\sym^3(\pi_p)$ is of Type I or II if $\varepsilon(\sym^3(\pi) \otimes \chi_p)= \chi_p(M') \chi_p'(s) \varepsilon(\sym^3 (\pi))$.
				\item 
				$\sym^3(\pi_p)$ is of Type III if $\varepsilon(\sym^3(\pi) \otimes \chi_p)= \chi_p(dM') \chi_p'(s_2)\varepsilon(\sym^3 (\pi))$.
			\end{enumerate}
			The above classification is also valid for $p=3$ with $a(\varkappa^3)>1$. For $\varkappa^3$ unramified, we always have $\varepsilon(\sym^3(\pi) \otimes \chi_3)= \varkappa^3(3)\chi_3'(s_2) \chi_3(M')\varepsilon(\sym^3 (\pi))$.  For $a(\varkappa^3)=1$, $\sym^3(\pi_3)$ is either Type I or II if $\varepsilon(\sym^3(\pi) \otimes \chi_3)= \chi_3'(s_2) \chi_3(M')\varepsilon(\sym^3 (\pi))$ and $\varkappa^3(3) \neq 1$.
			
			\vskip 2mm
			
			\noindent If $N_p$ is odd, then the following are true.
			\begin{enumerate}
				\item
				If $C_p \leq 1$, then ${\rm sym^3}(\pi_p)$ cannot be of Type I. 
				\item 
				If $p \geq 5$, then ${\rm sym^3}(\pi_p)$ cannot be of Type III. 
			\end{enumerate}
			In particular, for primes $p\geq 5$ with $C_p \leq 1$, ${\rm sym^3}(\pi_p)$ is always of Type II. 
			
			\noindent For $p=3$ with $N_3$ odd, we have the below classification of $\sym^3(\pi_3)$ and $K/\Q_3$. 
			\begin{itemize}
				\item 
				If $a(\varkappa^3)=0$, then $\sym^3(\pi_3)$ is of Type III. In this case, we have $\varepsilon(\sym^3(\pi) \otimes \chi_3)=  \chi_3(M')\varepsilon(\sym^3 (\pi))$. 
				\item
				If $a(\varkappa^3)=1$, then $\sym^3(\pi_3)$ is either Type I or II.
				\item 
				If $a(\varkappa^3) \geq 2$ is even, then it is always of either Type I or II and we have $\varepsilon(\sym^3(\pi) \otimes \chi_3)= \chi_3(M')\varepsilon(\sym^3 (\pi))$.
				\item 
				If $a(\varkappa^3) \geq 3$ is odd, then the corresponding ramified extensions are determined as follows depending upon the types.
				\begin{table}[htbp]
					\centering
					\begin{tabular}{ |l|l|l| }
						\hline
						{\bf Type of {\boldmath $\sym^3(\pi_3)$}}  & ${\bf K/\Q_3}$ & {\bf \ \ \ \ \ \ \ \ \ \ \ Condition} \\
						\hline
						\multirow{2}{*}{Type I or II} & $K=\Q_3(\sqrt{3})$ & $\varepsilon(\sym^3(\pi) \otimes \chi_3)=  \chi_3(M')\varepsilon(\sym^3 (\pi))$\\  & {$K=\Q_3(\sqrt{-3})$} & $\varepsilon(\sym^3(\pi) \otimes \chi_3)= - \chi_3(M')\varepsilon(\sym^3 (\pi))$ \\
						\hline
						\multirow{2}{*}{Type III} & $K=\Q_3(\sqrt{3})$ & $\varepsilon(\sym^3(\pi) \otimes \chi_3)=  \chi_3(dM')\varepsilon(\sym^3 (\pi))$\\  & {$K=\Q_3(\sqrt{-3})$} & $\varepsilon(\sym^3(\pi) \otimes \chi_3)= - \chi_3(dM')\varepsilon(\sym^3 (\pi))$ \\
						\hline 
					\end{tabular}
					\vskip 2mm
					\caption{Type of $\sym^3(\pi_3)$}
					\label{table555}
				\end{table}
			\end{itemize}
		\end{enumerate}
	\end{cor}

	\begin{proof}
		We first establish the relation $\varepsilon(\sym^3(\pi) \otimes \chi_p)=\chi_p(M')\varepsilon(\sym^3(\pi) \varepsilon_p$. For a prime $q (\neq p)$, the number $\varepsilon_q=\left(\frac{q}{p}\right)^{{\rm val}_q(a({\sym}^3(\pi)))}$ by Theorem \ref{q}. Note that $\prod_{q, q \neq p} \left(\frac{q}{p}\right)^{{\rm val}_q(a({\sym}^3(\pi)))} = \chi_p(M')$, where $M'$ denotes the exact power of $p$ that divides $a({\sym}^3(\pi))$. Now, from the variance number, we have $\varepsilon \left(\sym^3(\pi_{f,p}) \otimes \chi_p \right) =  \varepsilon \left(\sym^3(\pi_{f,p}) \right) \varepsilon_p$. We run through all primes $p$ and get that $\varepsilon \left(\sym^3(\pi) \otimes \chi_p \right) =  \varepsilon \left(\sym^3(\pi) \right) \prod_{p}\varepsilon_p = \chi_p(M') \varepsilon \left(\sym^3(\pi) \right) \varepsilon_p$. Now, we complete the proof applying Theorem \ref{supthm}, Prop. \ref{types} and Lemma \ref{phi}.
	\end{proof}

	\subsection{The case $p=2$:}
	For $p=2$, more representations of the Weil group are involved and it can be non-dihedral. Note that the non-dihedral supercuspidal case can occur only in $8$ cases with $N_2 \in \{3,4,6,7\}$. In this present article, we assume that $p=2$ is a dihedral supercuspidal prime where $\pi_2$ is induced from a quadratic extension $K/\Q_2$. There are seven posibilities of $K=\Q_2(\sqrt{t})$ with $t=-3,-1,3,2,-2,6,-6$. Among them $\Q_2(\sqrt{-3})$ is unramified and rest of them are ramified. Among ramified extensions, two of them (corresponding to $t=-1,3$) have discriminant with valuation $2$ and rest of them have discriminant with valuation $3$.
	Let $\de$ denote the valuation of the discriminant of $K/\Q_2$. 
	Thus, we have $\de \in \{ 2,3 \}$.
	Note that $\pi=1+\sqrt{t}$ is a uniformizer of $K=\Q_2(\sqrt{t})$, when $t=-1,3$ and $\pi=\sqrt{t}$ 
	is a uniformizer of $K$, when $t=2,-2,6,-6$.
	The relation~(\ref{indconductor}) for $p=2$ gives us that
	\begin{eqnarray} 
	\label{evenodd}
	N_2 = 
	\begin{cases}
	2 a(\varkappa), & \quad \text{if} \,\, K/\Q_2 \,\,
	\text{is unramified}, \\
	2+a(\varkappa), & \quad \text{if} \,\, K/\Q_2 \,\,
	\text{is ramified with valuation} \,\, 2, \\
	3+a(\varkappa), & \quad \text{if} \,\, K/\Q_2 \,\,
	\text{is ramified with valuation} \,\, 3.
	\end{cases}
	\end{eqnarray}
	
	As $\pi_2$ is dihedral, we also have the types of $\sym^3(\pi_2)$ as described in \S \ref{symtype}.  For Type I and II representations, using Equations \ref{indconductor} and \ref{eq2} we obtain that
	\begin{equation} \label{asym3p=2}
	a(\sym^3(\pi_2)) =
	\begin{cases}
	2a(\varkappa^3)+2a(\varkappa\epsilon_2'), 
	\quad \quad  \,\,\text{if} \,\, K/\Q_p \text{ is unramified},  \\
	a(\varkappa^3)+a(\varkappa\epsilon_2')+4, 
	\quad \,\,\,\text{if} \,\, K/\Q_p \text{ is ramified with valuation } 2, \\
	a(\varkappa^3)+a(\varkappa\epsilon_2')+6, 
	\quad \,\,\,\text{if} \,\, K/\Q_p \text{ is ramified with valuation } 3.
	\end{cases}
	\end{equation}
	
	\begin{prop} \label{a2}
		Let $\pi_2=\Ind_{W(K)}^{W(\Q_2)} (\varkappa)$ be a supercuspidal representation attached to a $2$-minimal form $f$. 
		We assume $C_2 \neq \frac{N_2}{2}$ if $K/\Q_2$ is unramified, and $a(\varkappa)\geq \de+1$ if $K/\Q_2$ is ramified. Then $a(\sym^3(\pi_2)) =2N_2$. Furthermore, $\sym^3(\pi_2)$ cannot be of Type III with $K/\Q_2$ ramified.
	\end{prop}
	
	\begin{proof}
		Let $\sym^3(\pi_2)$ be of Type I or II. If $K/\Q_2$ is unramified, then $a(\sym^3(\pi_2)) =2a(\varkappa^3)+2a(\varkappa \epsilon_2')=2N_2$ by the similar calculations used in the proof of Prop. \ref{prop7}. When $K/\Q_2$ is ramified, we assume $a(\varkappa)\geq \de+1$. For a minimal newform, using \cite[Prop. 5.10]{bm} we have that $a(\varkappa) \geq 3$ is odd (resp. $a(\varkappa)\geq 4$ is even) if the discriminant valuation is $2$ (resp. $3$). Combining Equ. \ref{evenodd} and \cite[Prop. 2.8]{LW}, we obtain that $C_2 \leq \frac{a(\varkappa)+1}{2}$ or $\frac{a(\varkappa)+2}{2}$ if the discriminant valuation is $2$ or $3$ respectively. Now, using Lemma \ref{tunnel} we conclude $a(\epsilon_2')< a(\varkappa)$.  Thus, $a(\varkappa \epsilon_2')=a(\varkappa)$. As $p=2$, Equations \ref{evenodd} and \ref{asym3p=2} imply that $a(\sym^3(\pi_2)) =2N_2$.
		
		Let $\sym^3(\pi_2)$ be of Type III, i.e., $\sym^3(\pi_2)=\varphi \oplus \varphi \omega_{K/\Q_2} \oplus {\rm Ind}^{W(\Q_2)}_{W(K)} (\varkappa \epsilon_2')$ for some character $\varphi$ of $\Q_2^\times$ satisfying $\varkappa^3=\varphi \circ N_{K/\Q_2}$. If $K/\Q_2$ is unramified, then as $a(\varphi)=a(\varkappa^3)$, it follows that $a(\sym^3(\pi_2)) =2N_2$. Now, assume that $K/\Q_2$ is ramified. Using Lemma \ref{tunnel}, $a(\varkappa^3)=a(\varphi)+a(\varphi \omega_{K/\Q_2})-\de$. By assumption, $a(\varkappa^3)\geq \de+1$. Hence, we conclude that $a(\varphi) \geq \de+1$. Therefore, $a(\varkappa)=a(\varkappa^3)=2a(\varphi)-\de$ is even or odd according as $\de=2$ or $3$ respectively. This contradicts the $2$-minimality of $f$ by \cite[Prop. 5.10]{bm}.
	\end{proof}

	
	Now, we calculate epsilon factors of $\varkappa\epsilon_2'\chi_2', \varkappa\epsilon_2'$ and $\chi_2'$ when $N_2=2$ with the assumption that $C_2 \neq \frac{N_2}{2}$. Here $\chi_2$ is as in \S \ref{aim}. This makes $\epsilon_2'$ unramified. Note that $N_2=2$ is possible only when $K/\Q_2$ is unramified and $a(\varkappa)=1$. Denote by $\omega$ a primitive cube root of unity.
	\begin{lemma} \label{chip'}
		For $N_2=2$, we have 
		\[
		\varepsilon(\varkappa\epsilon_2'\chi_2',\phi_K) = 2^{-4} \epsilon_2'^2(2) \varkappa^2(2) [-2i+\omega^{n}(-2\sqrt{2}+2\sqrt{2}i)];
		\]
		 where $\varkappa(\omega)=\omega^n, n=1,2$. 
		 Also, $\varepsilon(\chi_2',\phi_K) = 2^{-4} [-2\sqrt{2}+(2\sqrt{2}-2)i]$.
	\end{lemma}

	\begin{proof}
		Recall that $\chi_2'=\chi_2 \circ N_{K/\Q_2}$. As $K/\Q_2$ is unramified quadratic, $\chi_2'$ is quadratic and $a(\chi_2')=3$. Note that $\mco_K^\times/U_K^3$ is a group of order $48$ having representatives determined by $\mco_K^\times/U_K^1=\{1, \omega, \omega^2\}$, $U_K^1/U_K^2=\{1, 1+2\sqrt{-3}, 3+2\sqrt{-3}, 3\}$ and $U_K^2/U_K^3=\{1, 1+4\sqrt{-3}, 3+4\sqrt{-3}, 5\}$ where $\omega$ is a primitive cube root of unity in $\mco_K^\times$ and it exists by Hensel's lemma.
		
		Note that $a(\varkappa\epsilon_2'\chi_2')=3$. Now, proceeding as in the proof of \cite[Lemma 2.3]{bm}, we obtain that 
		\begin{equation} \label{eqalpha22}
		\varepsilon(\varkappa\epsilon_2'\chi_2', \phi_K) \overset{(\epsilon 2)}= \epsilon_2'(2)^2 \varepsilon(\varkappa \chi_2', \phi_K) 
		=\epsilon_2'^2(2) 2^{-4} \varkappa^2(2) \tau(\varkappa\chi_2', \phi_2\circ \mathrm{Tr}).
		\end{equation}
		By the choice of additive charcaters, $\tau(\varkappa\chi_2', \phi_2\circ \mathrm{Tr})= \sum_{x \in \mco_K^\times/U_K^3}$ $(\varkappa\chi_2')^{-1}(x) \phi_2(\mathrm{Tr}(\frac{x}{8}))$. As $a(\varkappa)=1$, there are two possibilities of $\varkappa$ sending $\omega$ to $\omega$ or $\omega$ to $\omega^2$. For $\omega \mapsto \omega^n, n=1, 2$, one can compute that $\tau(\varkappa\chi_2', \phi_2\circ \mathrm{Tr})=-2i+\omega^{n}(-2\sqrt{2}+2\sqrt{2}i)$. Hence, the result follows from \eqref{eqalpha22}.  
		
		In a similar way as above, we can compute $\varepsilon(\chi_2',\phi_K)$.
	\end{proof}

	\begin{lemma} \label{chip'11}
		For $N_2=2$, we have $\varepsilon(\varkappa\epsilon_2', \phi_K) = \frac{1}{2}$.
	\end{lemma}

	\begin{proof}
		As $a(\varkappa)=1$ and $\epsilon_2'$ is unramified, $a(\varkappa\epsilon_2')=1$. Proceeding as in the previous lemma,
		$
		\varepsilon(\varkappa\epsilon_2', \phi_K) \overset{(\epsilon 2)}=  \varepsilon(\varkappa, \phi_K) = 2^{-2} \tau(\varkappa, \phi_2\circ \mathrm{Tr}).
		$
		Now, $\tau(\varkappa, \phi_2\circ \mathrm{Tr})= \sum_{x \in \mco_K^\times/U_K^1}$ $\varkappa^{-1}(x) \phi_2(\mathrm{Tr}(\frac{x}{2}))$. As $a(\varkappa)=1$, there are two possibilities of $\varkappa$ sending $\omega$ to $\omega$ or $\omega$ to $\omega^2$.
		For $\omega \mapsto \omega$, we have
		\begin{itemize}
			\item 
			$\varkappa^{-1}(1) \phi_2(\mathrm{Tr}(\frac{1}{2}))=\phi_2(1)=e^{2\pi i}=1$ \vskip 1mm
			\item 
			$\varkappa^{-1}(\omega) \phi_2(\mathrm{Tr}(\frac{\omega}{2}))= \omega^2 \phi_2(\frac{-1}{2}) = \omega^2 e^{i\pi}=-\omega^2$ \vskip 1mm
			\item 
			$\varkappa^{-1}(\omega^2) \phi_2(\mathrm{Tr}(\frac{\omega^2}{2}))= \omega \phi_2(\frac{-1}{2})=\omega e^{i\pi}=-\omega$. 
		\end{itemize} 
		So $\tau(\varkappa, \phi_2\circ \mathrm{Tr})=2$. The same holds when $\varkappa(\omega)=\omega^2$. Hence, $\varepsilon(\varkappa\epsilon_2', \phi_K)=\frac{1}{2}$.
	\end{proof}
	
	Let us define two quantities
	\[d_1:=\frac{(\varkappa^{11} \epsilon_2'^2)(2)}{2^7} [-2\sqrt{2}+(2\sqrt{2}-2)i] [-2i+\omega^{n}(-2\sqrt{2}+2\sqrt{2}i)]\]
	\[d_2:= \frac{(\varkappa^{11} \epsilon_2'^2 \omega_{K/\Q_2})(2)}{2^5} [-2i+\omega^{n}(-2\sqrt{2}+2\sqrt{2}i)]\]
	where the integer $n$ is determined by $\varkappa(\omega)=\omega^n$.
	While computing the variance number $\varepsilon_2$ in the theorem below, we consider the twisting character $\chi_{-1}$ except when $N_2=2$; in which case we consider $\chi_2$. These characters are described in \S \ref{aim}. Following Lemma \ref{relation between multiplicative and additive character}, when $a(\varkappa)>3$ we define $s_1, s_2$ as in \eqref{character1}, \eqref{character2} respectively and we define $d_1, d_2$ as in \eqref{character}. Recall, $s=s_1 s_2$ and $d=d_1 d_2$. 
	
	\begin{theorem}
		Let $p=2$ be a dihedral supercupidal prime for a minimal form $f$ with $\pi_2=\Ind_{W(K)}^{W(\Q_2)} (\varkappa)$ and $K/\Q_2$ quadratic. If $K$ is unramified, then we have:
		\begin{enumerate}
			\item 
			If $a(\varkappa)>3$, then
			\begin{equation} \label{2unramified}
			\varepsilon_2 =
			\begin{cases}
			\chi_{-1}'(s), 
			\quad \quad \quad \,\,\,\,\, \text{if} \,\, \mathrm{sym}^3(\pi_2) \text{ is either Type I or II},  \\
			\chi_{-1}(d) \chi_{-1}'(s_2),
			\quad \text{if} \,\, \mathrm{sym}^3(\pi_2) \text{ is Type III}. 
			\end{cases}
			\end{equation}
			For $a(\varkappa)=1$, we have
			\begin{equation*} \label{2unramified1}
			\varepsilon_2 =
			\begin{cases}
			d_1, 
			\text{if} \,\, \mathrm{sym}^3(\pi_p) \text{ is either Type I or II},  \\
			d_2,
			~~ \text{if} \,\, \mathrm{sym}^3(\pi_p) \text{ is Type III}. 
			\end{cases}
			\end{equation*}
			\item 
			Let $K/\Q_2$ be ramified with discriminant valuation $\de$. If $a(\varkappa)\geq \de+1$, then
			\begin{equation} \label{2ramified}
			\varepsilon_2 =
			\begin{cases}
			1, \quad \text{if} \,\, \de=2,  \\
			\chi_{-1}'(s), \quad \text{if} \,\, \de=3.
			\end{cases}
			\end{equation} 
		\end{enumerate}
	\end{theorem}

	\begin{proof}
		Let $K/\Q_2$ be unramified. When $a(\varkappa)>3$, the proof is similar to the cases of odd primes. Now, assume $a(\varkappa)=1$. Note that $\varkappa^3, \varphi$ and $\omega_{K/\Q_2}$ are unramified. By the property $(\epsilon 2)$ of epsilon factors, we obtain that 
		\begin{enumerate}
			\item 
			$\frac{\varepsilon(\varkappa^3 \chi_2', \phi_K, dx)}{\varepsilon(\varkappa^3, \phi_K, dx)}=\varkappa^9(2)\varepsilon(\chi_2', \phi_K) \overset{\text{Lemma } \ref{chip'}}= \varkappa^9(2) 2^{-4} [-2\sqrt{2}+(2\sqrt{2}-2)i]$ \vskip 1mm
			\item 
			$\frac{\varepsilon(\varphi \omega_{K/\Q_2} \chi_2, \phi, dx) ~~\varepsilon(\varphi \chi_2, \phi, dx)}{\varepsilon(\varphi \omega_{K/\Q_2}, \phi, dx) ~~\varepsilon(\varphi, \phi, dx)}=\varphi^6(2) \omega_{K/\Q_2}(2) \{\varepsilon(\chi_2, \phi)\}^2 \overset{\text{Lemma } \ref{alpha}}= \frac{\varkappa^9(2) \omega_{K/\Q_2}(2)}{4}$. 
		\end{enumerate}
		Also, by Lemmas \ref{chip'} and \ref{chip'11}, $\frac{\varepsilon(\varkappa \epsilon_2' \chi_2', \phi_K, dx)}{\varepsilon(\varkappa \epsilon_2', \phi_K, dx)}=2^{-3} \epsilon_2'^2(2) \varkappa^2(2) [-2i+\omega^{n}(-2\sqrt{2}+2\sqrt{2}i)]$.
		Now using \eqref{unr} and \eqref{unr1} we get the desired result.
		
		\medskip 
		
		For $K/\Q_2$ ramified with $a(\varkappa) \geq \de+1$, recall from Lemma \ref{a2} that $\sym^3(\pi_2)$ cannot be of Type III, so we consider Type I or II. First, we consider $K/\Q_2$ with $\delta=2$. As $a(\chi_{-1})=2$ and $a(\omega_{K/\Q_2})=2$, by Lemma \ref{tunnel}, $a(\chi_{-1}')=a(\chi_{-1})+a(\chi_{-1} \omega_{K/\Q_2})-a(\omega_{K/\Q_2})=0$. Thus, by the property $(\epsilon 2)$, 
		\begin{enumerate}
			\item 
			$\varepsilon(\varkappa^3\chi_{-1}', \phi_K)=\chi_{-1}'(\pi)^{a(\varkappa^3)-1} \varepsilon(\varkappa^3, \phi_K)$ 
			\item 
			$\varepsilon(\varkappa \epsilon_2'\chi_{-1}', \phi_K)=\chi_{-1}'(\pi)^{a(\varkappa \epsilon_2')-1} \varepsilon(\varkappa \epsilon_2', \phi_K)$. 
		\end{enumerate}
		Therefore, using Equ. \ref{unr}, $\varepsilon_2=\chi_{-1}'(\pi)^{a(\varkappa^3)+a(\varkappa \epsilon_2')}=1$, follows from the fact $a(\varkappa^3)=a(\varkappa)=a(\varkappa \epsilon_2')$. For the fisrt equality of conductors we use $a(\varkappa)\geq \de+1$ and for the last equality see the proof of Prop. \ref{a2}. Next, we consider $K/\Q_2$ with $\delta=3$. As $a(\chi_{-1})=2$ and $a(\omega_{K/\Q_2})=3$ , we have $a(\chi_{-1}')=a(\chi_{-1})+a(\chi_{-1} \omega_{K/\Q_2})-a(\omega_{K/\Q_2})=2$. Note $a(\varkappa^3)=a(\varkappa)$ and $a(\varkappa\epsilon_2')=a(\varkappa)$. By Theorem \ref{epsilon factor while twisting}, we obtain that $\varepsilon(\varkappa^3\chi_{-1}', \phi_K)=\chi_{-1}'(s_1) \varepsilon(\varkappa^3, \phi_K)$ and $\varepsilon(\varkappa \epsilon_2'\chi_{-1}', \phi_K)=\chi_{-1}'(s_2) \varepsilon(\varkappa \epsilon_2', \phi_K)$. Therefore, using \eqref{unr}, $\varepsilon_2=\chi_{-1}'(s)$. 
	\end{proof}
The below theorem is helpful to classify the types of $\sym^3(\pi_2)$ when $\pi_2$ is dihedral and it can proved by the similar argument used in Theorem \ref{q}.
\begin{theorem}\label{2}
	Let $q \neq 2$. Then, $\varepsilon_q = \chi_{-1}(q)^{{\mathrm{val}}_q(a({\rm sym}^3(\pi)))}$.
\end{theorem}

	\begin{cor} \label{mainkorop=2}
		Let $\pi_2$ be the dihedral supercuspidal representation attached to a $2$-minimal cusp form $f$. Then $\pi_2= \Ind_{W(K)}^{W(\mathbb{Q}_2)}(\varkappa)$ with $K/\Q_2$ quadratic and we have the following:
		\begin{enumerate}
			\item If $N_2$ is even, then $K/\Q_2$ is unramified. In this case, for $N_2>6$ we have
			\begin{enumerate}
				\item ${\rm sym^3}(\pi_2)$ is of Type I or II if $\varepsilon({\rm sym^3}(\pi)\otimes \chi_{-1})=\chi_{-1}(M')\chi'_{-1}(s)\varepsilon({\rm sym^3}(\pi))$. 
				
				\item ${\rm sym^3}(\pi_2)$ is of Type III if $\varepsilon({\rm sym^3}(\pi)\otimes \chi_{-1})=\chi_{-1}(dM')\chi'_{-1}(s_2)\varepsilon({\rm sym^3}(\pi))$.
			\end{enumerate}
			\item If $N_2 \geq 5$ is odd, then $K/\Q_2$ is ramified. In this case, if $a(\varkappa)\geq \de+1$, then ${\rm sym^3}(\pi_2)$ can not be of Type III. Furthermore, if $\chi_{-1}'(s)=-1$, then 
			\begin{enumerate}
				\item $K/\mathbb{Q}_2$ has discriminant valuation $\de =2$ if $$\varepsilon({\rm sym^3}(\pi)\otimes \chi_{-1})=\chi_{-1}(M')\varepsilon({\rm sym^3}(\pi)).$$
				
				\item $K/\mathbb{Q}_2$ has discriminant valuation $\de=3$ if $$\varepsilon({\rm sym^3}(\pi)\otimes \chi_{-1})=-\chi_{-1}(M')\varepsilon({\rm sym^3}(\pi)).$$
			\end{enumerate}
			We cannot conclude anything if $\chi_{-1}'(s)=1$.
		\end{enumerate}
	\end{cor}
	\begin{proof}
		For a $2$-minimal form $f$, we get the classification of $K/\Q_2$ depending upon $N_2$ is even or odd by \cite[Prop. 5.10]{bm}. The rest can be easily checked by the same argument used for odd primes.
	\end{proof}

	\section{Conductors of symmetric cube transfers} \label{symconductor}
	We consider the following Hypothesis {\bf (H)} to hold for a minimal newform $f\in S_k(N, \epsilon)$. 
	
	Hypothesis {\bf (H)}: (H1) For the supercuspidal prime $p=2$, $\pi_{f, 2}$ is dihedral. 
	\begin{enumerate}
		\item [(H2)] 
		If $p=2$ is a ramified dihedral supercuspidal prime for $f$, then $N_2\geq 2\de+1$.
		\item [(H3)]
		If $p$ is an unramified dihedral supercuspidal prime for $f$, then $C_p \neq \frac{N_p}{2}$.
	\end{enumerate}

\begin{definition} \label{sets} Let $P, SP, SC$ denote the set of primes with $\pi_p$ is of either principal series, special or supercuspidal type respectively. Let $\omega_p$ be as in \S~\ref{localprop}. Note that if $p \in SC$ then $\pi_p=\Ind_{W(K)}^{W(\Q_p)}(\varkappa)$. Let us consider the following sets:
\begin{equation}
	P_1:=\{p \in P: p \geq 5 \text{ with } N_p>1 \text{ or } N_p=1, \circ(\omega_p|_{\Z_p^\times})>3\}.
\end{equation}
\begin{equation}
P_2:=\{p \in P: (i)~p \notin P_1,  \text{ or } (ii)~ N_3 \neq 2 \text{ or } N_3=2, \circ(\omega_3|_{\Z_3^\times})\neq 3, \text{ or } (ii)~ N_2>3\}.
\end{equation}
\begin{equation}
P_3:=P \setminus (P_1 \cup P_2)=\{p \in P: N_3=2 \text{ with }\circ(\omega_3|_{\Z_3^\times})= 3 \text{ or } N_2 \leq 3\}.
\end{equation}
\begin{equation}
S_1:=\{p \in SC: p \geq 5, N_p=2, \circ(\varkappa|_{\mco_K^\times})=3\}
\end{equation}
\begin{equation}
S_2:=\{p \in SC: p \geq 5, p \notin S_1\}.
\end{equation}
\end{definition}

	\begin{theorem} \label{conductorthm}
		Let $\pi \cong \otimes_p' \pi_{p}$ be the cuspidal automorphic representation attached to a minimal newform $f \in S_k(N, \epsilon)$ satisfying Hypothesis {\bf (H)}. Then the conductor of $\sym^3(\pi)$ is given by
		 \begin{eqnarray}  \nonumber
			a \left( \sym^3(\pi) \right)= N  \prod_{p\in SP \cup P_1} p^{2N_p} \prod_{p\in P_2} p^{2N_p-1}  \prod_{p\in P_3 \cup S_2} p^{N_p}
			\prod_{2 \in SC} 2^{N_2} \prod_{3\in SC} 3^{e_\varkappa}.
			\end{eqnarray} 
	\end{theorem}
	\begin{proof}
		We have $a(\pi)=\prod_p p^{a(\pi_p)}$ and $a(\sym^3(\pi))=\prod_p p^{a(\sym^3(\pi_p))}$. Note that $a(\pi_p)=N_p$, the exact power of $p$ that divides $N$. Now, combining Propositions \ref{aprin}, \ref{aspecial}, \ref{prop7}, \ref{ap=3} and \ref{a2}, we obtain the theorem.
	\end{proof}
	
	\begin{remark} \label{r1}
		If $f \in S_k(\Gamma_0(N))$ is a minimal cusp form with the corresponding automorphic representation $\pi$, then from Theorem \ref{conductorthm} we deduce that
		\begin{eqnarray}  \label{eq3} 
		a \left( \text{sym}^3(\pi) \right)= N \prod_{p \in SP} p^2   \prod_{p\in S_2} p^{N_p}
		\prod_{2 \in SC} 2^{N_2} \prod_{3\in SC} 3^{e_\varkappa}.
		\end{eqnarray}
		Let $f_E \in S_2(\Gamma_0(N))$ denote the newform attched to an elliptic curve $E$ of conductor $N$ by the modularity theorem. We now show that the formula of $a \left( \text{sym}^3(\pi_{f_E}) \right)$ (cf.  \cite[Corollary 3.2.3]{manami}) can be recovered from Equ. \ref{eq3}.

		Note that if $f_E$ is minimal, then there is no prime $p$ at which $E$ has additive, potentially good reduction over an abelian extension (which corresponds to the ramified principal series type of $\pi_{f_E, p}$) \cite[\S 2]{MR3729491}. If $E$ has multiplicative reduction or additive, potententially multiplicative reduction at $p$, then $\pi_{f_E, p}$ is of special type with 
		$a(\sym^3(\pi_{f_E, p}))=3$ \cite[Table 3]{manami} and it matches with ours [cf. Prop.~\ref{aspecial}]. Also, if $E$ has additive, potentially good reduction over a non-abelian extension then $\pi_{f_E, p}=\Ind_{W(K)}^{W(\Q_P)} \varkappa$ with $K/\Q_p$ quadratic. 
		In this case, if $p \geq 5$ then $N_p=2$ with $\varkappa^i|_{\mco_K^\times}=1$ for $i=3, 4, 6$  and we have $a(\sym^3(\pi_{f_E, p}))=2$ if $\varkappa^3|_{\mco_K^\times}=1$; otherwise, it is $4$  \cite[Table 4]{manami}. This follows from our Prop.\ref{prop7}. For $p=3$, if $K/\Q_3$ is unramified, then $\varkappa^i|_{\mco_K^\times}=1$ with $i=3, 4, 6$. From Table 5 of {\it loc. cit.}, $N_3=2$ if $i=4$; otherwise $N_3=4$. Thus, we compute $e_\varkappa=0$ if $i=3$;  otherwise $e_\varkappa=2$. 
		On the other hand, if $K/\Q_3$ is ramified, then $N_3=3$ or $5$ with $\varkappa^i|_{\mco_K^\times}=6$ and we have $a(\sym^3(\pi_{f_E, p}))=5$ or $7$ respectively. We compute $e_\varkappa=2$ from \eqref{echi}.

		Hence, the conclusion of Prop. \ref{ap=3}, that is, $a(\sym^3(\pi_{f_E, p}))=N_3+e_\varkappa$ coincides with that of Table 5 of {\it loc. cit.}. Note $p=2$ cannot be a supercuspidal prime for $f_E$ by assumption of \cite[Corollary 3.2.3]{manami}.
		As a result, we deduce that our result on the conductor of the symmetric cube transfer of a cusp form with trivial nebentypus [cf. Theorem \ref{conductortheorem}] coincides with \cite[Corollary 3.2.3]{manami}.
	\end{remark}

\end{document}